\documentclass[12pt]{amsart}

\usepackage[utf8]{inputenc}
\usepackage{fullpage}
\usepackage[pagewise]{lineno}
\usepackage[colorlinks,linkcolor=blue,citecolor=blue]{hyperref}
\usepackage{graphicx}
\usepackage{amsmath}
\usepackage{amsthm}
\usepackage{amsfonts}
\usepackage{amssymb}
\usepackage{mathtools}
\usepackage{array}   
\newcolumntype{L}{>{$}l<{$}}

\usepackage{pgf,tikz}
\usepackage{mathrsfs}
\usetikzlibrary{arrows}
\newtheorem{Thm}{Theorem}[section]
\newtheorem{thm}{Theorem}[subsection]
\newtheorem{cor}[thm]{Corollary}
\newtheorem{prop}[thm]{Proposition}
\newtheorem{lem}[thm]{Lemma}
\newtheorem{conj}[Thm]{Conjecture}

\newtheorem{defn}[thm]{Definition}

\theoremstyle{definition}

\theoremstyle{remark}
\newtheorem{ques}[Thm]{Question}
\newtheorem{rem}[thm]{Remark}

\usepackage{eucal}
\usepackage{graphicx}
\usepackage{multirow}
\usepackage[all,knot]{xy}

\DeclareMathOperator{\Mod}{Mod}
\DeclareMathOperator{\PGL}{PGL}
\DeclareMathOperator{\PSL}{PSL}
\DeclareMathOperator{\SL}{SL}
\DeclareMathOperator{\GL}{GL}
\DeclareMathOperator{\SMod}{SMod}
\DeclareMathOperator{\End}{End}
\DeclareMathOperator{\Add}{Add}
\DeclareMathOperator{\Gal}{Gal}
\DeclareMathOperator{\Tr}{Tr}
\calclayout
\newcommand{\RNum}[1]{\uppercase\expandafter{\romannumeral #1\relax}}

\newcommand{\Tet}[6]{\Big\langle\begin{matrix}
  #1 & #2 & #3 \\
  #4 & #5 & #6 
\end{matrix}\Big\rangle}
\newcommand{\Sym}[6]{\Big\{\begin{matrix}
  #1 & #2 & #3 \\
  #4 & #5 & #6 
\end{matrix}\Big\}}
\usepackage{graphicx}
\graphicspath{ {./graph/} }

\begin{document}
\title{Projective representations of Hecke groups from Topological quantum field theory}
\author{Yuze Ruan}
\email{ryz22@mails.tsinghua.edu.cn}

\maketitle
\begin{abstract}
  We construct projective (unitary) representations of Hecke groups from the vector spaces associated with the Witten-Reshetikhin-Turaev topological quantum field theory of higher genus surfaces. In particular, we generalize the modular data of Temperley-Lieb-Jones modular categories. We also study some properties of the representation. We show the image group of the representation is infinite at low levels in genus $2$ by explicit computations. We also show the representation is reducible with at least three irreducible summands when the level equals $4l+2$ for $l\geq 1$.     
\end{abstract}
\section{Introduction}
In 1983, Vaughan Jones discovered a new family of representations of braid groups, from the study of the index of II$_1$ subfactors \cite{Jones83, Jon87}, which in turn gave a beautiful new invariant for knots, known as the Jones polynomials. Later, Witten in his seminal paper \cite{Wit89} uncovered a deep connection between the Jones polynomial and the Chern-Simons gauge theory. He further provided arguments for constructing $2+1$ topological quantum field theory (TQFT) based on it. The first rigorous mathematical construction of TQFT was developed by  Reshetikhin and Turaev, using quantum groups $U_q(\mathfrak{sl}_2)$ \cite{RT91}. A skein theoretical construction was given by Blanchet, Habegger, Masbaum, and Vogel \cite{BHMV1, BHMV2} (pioneered by Lickorish \cite{Lickorish93}). Turaev carried out the most general construction in \cite{Tur94}, which employs modular categories as input data. 

In essence, a $2+1$ TQFT is a symmetric monoidal functor from the cobordism category to the category of modules over some ring (often equipped with additional structure),  we denote the module associated to a surface $\Sigma$ by $V(\Sigma)$, and denote the homomorphism associated to a cobordism $M$ by $Z(M)$. A fascinating feature of TQFTs is their ability to produce projective representations $\Mod(\Sigma)\to \PGL(V(\Sigma))$ of the mapping class groups $\Mod(\Sigma)$, These representations are known to be asymptotically faithful \cite{FKW02,Anderson06}. In particular, when $\Sigma$ is an $n$-punctured disc, the corresponding mapping class group is the braid group $B_n$, and the representations, derived from $U_q(\mathfrak{sl}_2)$, recover the Jones representations \cite{Jon87}. For a genus-1 surface, $\Mod(\Sigma)$ is isomorphic to $\SL_2(\mathbb{Z})$. The images of the generators 
$
s=\begin{pmatrix}
0&-1\\
1&0
    \end{pmatrix},\  
t=\begin{pmatrix}
1&1\\
0&1
\end{pmatrix}, 
$  yield the modular $S$- and $T$-matrices, referred as the modular data of the input modular categories. While it is not a complete invariant for modular categories \cite{MS21}, they encode many information and play a crucial role in the classification of the modular categories \cite{RSW09, Rank5}, as well as  in the classification of partition functions in the conformal field theory \cite{CIZ87}. Moreover, In \cite{NS10}, Ng and Schauenburg demonstrated that the kernel of this representation is a congruence subgroup of $\SL_2(\mathbb{Z})$, in particular, the image is finite. 

In this paper, we focus on the skein version of TQFT constructed in \cite{BHMV2}, or equivalently the TQFT derived from Temperley-Lieb-Jones (TLJ) categories \cite{Wang10}, \cite[Chapter~XII]{Tur10}. We generalize the modular data of TLJ-categories through the mapping class group representation of higher genus surfaces, and we find the representations of Hecke groups $\Gamma_q\in \PSL_2(\mathbb{R})$ ($\tilde{\Gamma}_q\in \SL_2(\mathbb{R})$), which are generated by $t_q=\begin{pmatrix}
1&2\cos(\frac{\pi}{q})\\
0&1
\end{pmatrix}, \ 
s=\begin{pmatrix}
0&-1\\
1&0
\end{pmatrix}$ ($q\geq 3$), and the image of $t_q$ is a diagonal matrix. 
\begin{Thm}\label{mainthm0}
We have projective (unitary) representations $h_r$ (where $r$ is the level of the theory) of $\tilde{\Gamma}_{2g+1}$ from the TQFT vector spaces $V_r(\Sigma_g)$. In particular, when $g=1$, $h_r(t_3),h_r(s)$ provides the modular data of TLJ modular categories. When $g=2$, we define 
$\mathcal{J}:=h_{r}(s)$ and 
a diagonal matrix $\mathcal{T}_r:=h_r(t_5)$. They  satisfy the relations: 
\begin{equation}
\begin{aligned}
\mathcal{J}_r^2&=I\ ,\\
(\mathcal{T}_r\mathcal{J}_r)^5&=(\frac{\mathcal{P}_r^+}{\mathcal{P}_r^-})^2I\ .
\end{aligned}
\end{equation}
\end{Thm}
It's natural to ask the following 

\begin{ques}
Is the image of $h_r$ finite for $g\geq 2$?
\end{ques}

We perform some concrete calculations for the genus-$2$ case and obtain the following results:
\begin{Thm}\label{iftyimg}
The group $h_r(\Gamma_5)$ is infinite for $r=3,7,9,11,13$.
\end{Thm}
And it appears that the trace of certain elements grows exponentially in $r$.
It is known from the result of Funar \cite{Funar99} that if one considers the entire mapping class group, then the image is infinite in almost all cases. Additionally, Masbaum found an infinite order element \cite{Mas99}. However, their methods cannot directly address our question. The reason is that they used the factorization axiom to cut the surface into smaller pieces, and studied the mapping classes supported on subsurfaces, therefore the problem can be reduced to the calculations of braid group representations. In our case, however, Pseudo-Anosov mapping classes are generic within the corresponding subgroup of the mapping class group. meaning they cannot be supported on any subsurface. We noticed the following conjecture of Andersen, Masbaum and Ueno.

\begin{conj}\cite[Conjecture~2.4]{AMU06}
The image of a Pseudo-Anosov mapping class under TQFT representations is of infinite order for all sufficiently large levels.
\end{conj}
 
If this conjecture holds, it would imply that most of our representations are infinite. Recent progress on this conjecture for higher-genus surfaces with at least two boundary components is discussed in \cite{DK19}, along with its connection to the volume conjecture in \cite{BDKY22}.

Here is a brief outline of this paper. In Section 2, we review the basic definitions and properties of the Hecke group and the mapping class group. In Section 3, we review Thurston's construction and use it to define an inclusion $\rho: \tilde{\Gamma}_{2g+1}\hookrightarrow\Mod(\Sigma_g)$ with an explicit geometric description of $\rho(s)$. In Section 4, we review the general framework of TQFT and the mapping class group actions. Here, we carefully analyze these actions to prove Theorem \ref{mainthm0} and perform explicit calculations to establish Theorem \ref{iftyimg}. In Section 5, we examine the reducibility of $h_r$. By using the spin structures, we conclude that the representation is reducible with at least three irreducible summands when $r=4l+2$ for $l\geq 1$.     

\section*{Acknowledgement}
The author would like to thank Vaughan Jones, this work cannot be done without his constant support, guidance, and
encouragement. The author thanks Dietmar Bisch for his constant support at Vanderbilt. The author thanks Spencer Dowdall for helpful discussions and providing the reference \cite{leininger2004}.  The author thanks Eric Rowell and Yilong Wang for helpful discussions. The author also thanks Zhengwei Liu and BIMSA (Beijing Institute of Mathematics Sciences and Applications) where this work was completed.
\begin{figure}[h]
\centering
\label{fig:surface}
\includegraphics[width=0.8\textwidth]{{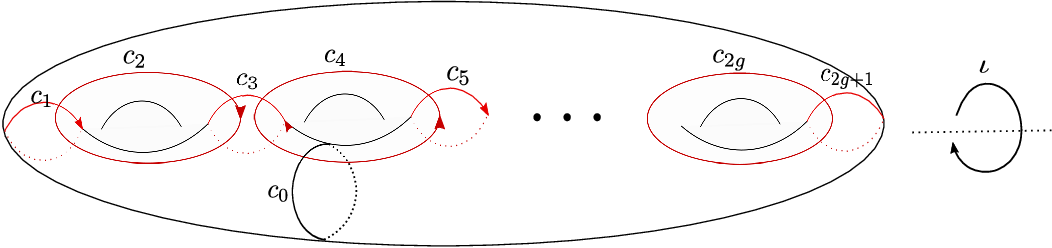}}
\caption{Genus $g$ surface with simple closed curves}
\end{figure}
\section{Preliminaries}
\subsection{Hecke group}
Here we mainly follow the discussion of Hecke groups in \cite[Appendix~III]{GHJ89}.    

\begin{defn}
The Hecke group $\Gamma_{q}\  (q\geq 3,\  odd)$ is the subgroup of $\PSL(2,\mathbb{R})$ generated by

$$ A_q=\begin{pmatrix}
1&\lambda\\
0&1
\end{pmatrix}\ , \ 
B_q=\begin{pmatrix}
1&0\\
-\lambda&1
\end{pmatrix}\ . \ \  (\lambda=2\cos\frac{\pi}{q})$$.
\end{defn}
\begin{thm}[] \cite{GHJ89}\label{Heckegroup}
 Let $J=\begin{pmatrix}
0&-1\\
1&0
\end{pmatrix}$, we have $J=A_q^{-1}(A_qB_q)^{\frac{q+1}{2}}$ and
$$\Gamma_q=<J, A_qJ>\cong \mathbb{Z}_2* \mathbb{Z}_{q}\ .$$
Moreover, if we view the matrices as elements in $\SL(2,\mathbb{R})$, then 
\begin{equation}\label{J:def}
J=(A_qB_q)^{\frac{q(q-1)}{2}}A_q^{-1}(A_qB_q)^{\frac{q+1}{2}}\ ,
\end{equation}
and they generate a group $\tilde{\Gamma}_q$ isomorphic to amalgamate free product $\mathbb{Z}_2*_{\mathbb{Z}_2} \mathbb{Z}_{q}$, which has the presentation 
$$\tilde{\Gamma}_q=<J, A_qJ>\cong <s,t|s^4=(ts)^{2q}=1\ ,\ s^2=(ts)^{q}>\ .$$
 
\end{thm}
\begin{proof}
One can find the proof for the structure of $\Gamma_q$ in \cite[Appendix~III]{GHJ89} or \cite[Chapter~II]{dlHarpe00}. The structure for $\tilde{\Gamma}_q$ follows easily from checking the relation and observing $J^2=(A_qB_q)^q=-I,\  (A_qJ)^2=-A_qB_q$.      
\end{proof}

\begin{rem}
When $q=3$,  $\Gamma_q\cong\PSL(2,\mathbb{R}),\ \tilde{\Gamma}_q\cong\SL(2,\mathbb{R}) $. When $\lambda\geq 2$, the group is freely generated by those two elements. 
\end{rem}

\subsection{Mapping class group}
This section mainly follows \cite{FM12,Bir74MCG}.
\begin{defn}
Let $\Sigma$ be a surface possible with punctures and boundaries, and $Homeo^+(\Sigma,\partial \Sigma)$ be the group of orientation-preserving homeomorphisms of $\Sigma$ that restrict to the identity on $\partial \Sigma$. The \textbf{mapping class group} of $\Sigma$, denoted $\Mod(\Sigma)$, is the group 
$$
\Mod(\Sigma)=Homeo^{+}(\Sigma,\partial \Sigma)/isotopy\ .
$$
\end{defn}
\begin{defn}\label{def:Dehntwist}
Fix a simple closed curve $\gamma$ on the surface, the \textbf{right\ (left) Dehn twist} about $\gamma$, denote by $T_{\gamma}$ $(resp.\ T^{-1}_{\gamma})$, is the isotopy class of a homeomorphism supported in an annular neighborhood $U$ of $\gamma$. More precisely, let $\psi: U\to \mathbb{R}/l_1\mathbb{Z}\times [0,l_2] $ be an orientation preserving homeomorphism, $T_\gamma$ is given by conjugating $\psi$ with the affine map
$$
(x,y)\mapsto(x\pm y\frac{l_1}{l_2},y)\ .
$$
Where $+$ gives a right (or positive) twist, while $-$ gives a left (or negative) twist.
\end{defn}

For now, we consider mainly closed surfaces with no punctures and right Dehn twists, the presentation of the mapping class group is known, for example see\cite{Wajnryb83}. The next proposition gives many interesting properties of Dehn twists, one can see, for example, \cite[Chapter~3]{FM12} for the proofs. 
\begin{prop}\label{DT_property}
Let $\gamma_1, \gamma_2$ be any isotopy classes of simple closed curves in $\Sigma$ with geometric intersection number $i(\gamma_1,\gamma_2)$. Let $f\in \Mod(\Sigma)$, we have \\
(a) $T_{f(\gamma_1)}=fT_{\gamma_1}f^{-1}$\ , \\
(b) $fT^k_{\gamma_1}=T^k_{\gamma_1}f\iff f(\gamma_1)=\gamma_1$\ ,\\
(c) $i(\gamma_1,\gamma_2)=0\iff T_{\gamma_1}T_{\gamma_2}=T_{\gamma_2}T_{\gamma_1}\ .$

\end{prop}

\begin{thm}[\cite{FM12,Wajnryb83}]
Let $\Sigma_g$ denote genus $g$ closed surface with no punctures, then $\Mod(\Sigma_g)$ is generated by the (left) Dehn twists around curves $c_i (0\leq i\leq 2g)$ shown in Figure \ref{fig:surface}. There are relations among generators: the disjoint relation (far commutativity), the braid relation, the chain relation, the lantern relation and the hyperelliptic relation.
\end{thm}

\begin{defn}[\cite{FM12}]
Let $\iota_g$ be the hyperelliptic involution in Figure 1 and let $SHomeo^+(\Sigma_g)$ be the centralizer in $Homeo^+(\Sigma_g)$ of $\iota_g$, i.e.,
$$
SHomeo^+(\Sigma_g)=C_{SHomeo^+(\Sigma_g)}(\iota_g)\ .
$$
The \textbf{symmetric mapping class group}, denoted by $\SMod(\Sigma_g)$, is the group
$$
\SMod(\Sigma_g)=SHomeo^+(\Sigma_g)/isotopy\ .
$$
\end{defn}
\begin{rem}
In general, a hyperelliptic involution is an order two mapping class acting on the homology by $-I$, and it is unique up to conjugations when genus $\geq 3$ \cite[Proposition~7.15]{FM12}. Here we pick the special $\iota_g$ as indicated in Figure 1. 
\end{rem}

\begin{thm}[Birman-Hilden Theorem]
For any $g$, $\SMod(\Sigma_g)/<\iota_g>\cong \Mod(S_{0,2g+2})$, where $\Mod(S_{0,2g+2})$ is the mapping class group of  $2g+2$ punctured sphere, or the spherical braid group $\pi_1B_{0,2g+2}S^2$. In particular, $\SMod(\Sigma_g)$ is generated by $T_{c_i}(1\leq i\leq 2g+1)$, they satisfy the relations in the braid group $\mathcal{B}_{2g+2}$. 
\end{thm}

\begin{figure}[b]
\centering
\includegraphics[width=0.5\textwidth]{{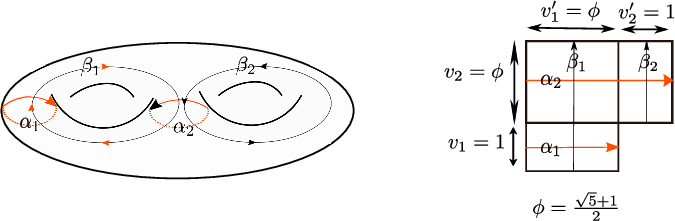}}
\caption{A flat structure of genus 2 surfaces where the given multi-twists act affinely.}
\label{fig:Fstructure}
\end{figure}  
  
\section{Hecke group inside mapping class group} 
\subsection{Flat structure and Thurston's construction}\cite{Th88,leininger2004,Mc06,FM12} 

 \begin{defn}
 A \textbf{multicurve} $A=\{\alpha_1,\cdots, \alpha_n\}$ is a set of pairwise disjoint simple closed curves on the surface, the \textbf{multi-twist} with \textbf{multiplicity} $ p:=\{p_1,\cdots,p_n\}$ ($p_i\in \mathbb{Z}_+$) about $A$ is a mapping class $T_{A,p}:=\prod\limits_{i=1}^nT_{\alpha_i}^{p_i}$. For simplicity, we write $T_A:=T_{A,\{1,\cdots,1\}}$. 
 \end{defn} 
  
  \begin{defn}
 A pair of multicurves $A=\{\alpha_1,\cdots \alpha_n\}$, $B=\{\beta_1\cdots \beta_m\}$ \textbf{bind} the surface $\Sigma_g$ if they meet only at transverse double points, and every component of $\Sigma_g-(A\cup B)$ is a polygonal
region with at least 4 sides (running alternately along $A$ and $B$).
\end{defn}
Now view $A\cup B$ as a (bipartite) graph, where the vertexes are given by the simple closed curves in $A,B$, and two vertexes are connected by $n$ edge iff the corresponding curves have geometric intersection number $n$. We denote the graph by $\mathcal{G}(A,B)$, and  the corresponding adjacency matrix by $N$, where $N_{i,j}=i(\alpha_i,\beta_j)$. Let $p=\{p_1, \cdots p_n\},\ q=\{q_1,\cdots q_m\} $ be a pair of multiplicities, we denote the corresponding diagonal matrix by $P, Q$.  If $A,B$ binds $\Sigma_g$, then $\mathcal{G}(A,B)$ is connected, hence the matrix $M=\begin{pmatrix}
0&PN\\
QN^t&0
\end{pmatrix}$ is primitive, and one can apply the Perron–Frobenius theorem. We denote the Perron–Frobenius eigenvalue and eigenvector by $\mu^{p,q}_{A,B}, \begin{pmatrix}
v^{p,q}_{A,B}\\
v'^{p,q}_{A,B}
\end{pmatrix}$, such that
\begin{equation}\label{eq:PerF}
\begin{aligned}
PNv'^{p,q}_{A,B}&=\mu^{p,q}_{A,B} v^{p,q}_{A,B}\ ,\\
QN^tv^{p,q}_{A,B}&=\mu^{p,q}_{A,B} v'^{p,q}_{A,B}\ .    
\end{aligned}
\end{equation} \\ 

Given a surface $\Sigma_g$, Thurston constructs a certain type of \textbf{flat structure} (singular Euclidean structure) on the surface \cite{Th88}, also see \cite{WrightAlex15} for equivalent definitions of the flat structure. Here we will give one arising from polygons.
\begin{defn}[\cite{WrightAlex15}]\label{Def:flatsurface}
 A \textbf{flat structure} on the surface $\Sigma_g$ is given by a cell decomposition consisting of a finite union of polygons in $\mathbb{C}$ (Euclidean polygons), with a choice of pairing of parallel sides of equal length. Two sets of polygons are considered to define the same flat structure if one can be cut into pieces along straight lines and these pieces can be translated and re-glued to form the other set of polygons. 
\end{defn}
For example Figure \ref{fig:Fstructure} gives a flat structure on $\Sigma_2$, where we have one vertex and it is the only 0-cell.\\
Now given a pair of multi curves $A,B$ binding $\Sigma_g$, we construct a flat structure on which the corresponding multi-twists $T_A, T_B$ act affinely.

We assign $\{|\alpha_1|,\cdots, |\alpha_n|\}=v^{p,q}_{A,B},\ \{|\beta_1|,\cdots,|\beta_m|\}=v'^{p,q}_{A,B}$ to be the length of curves in $A,B$. Consider the dual cell decomposition of the obvious cell decomposition coming from cutting along the curves in $A\cup B$ (since $A\cup B$ binds $\Sigma_g$), the number of $2$-cells is equal to the number of nonzero entries in $N$, and each $2$-cells is a rectangle. The length we assign to each $1$-cell is equal to the length of the curve intersecting with it, see Figure \ref{fig:Fstructure}. Using this flat structure we get an action of a pair of multi-twists with multiplicities.

\begin{thm}[Thurston's construction \cite{Mc06,Th88}]\label{Thconstruct}
Let $A=\{\alpha_1,\cdots \alpha_n\}$, $B=\{\beta_1\cdots \beta_m\}$ be a pair of multi curves binding the surface $\Sigma_g$,  and $p=\{p_1, \cdots p_n\},\ q=\{q_1,\cdots q_m\} $ be a pair of multiplicities. Then we have a representation $\rho_{A,B}^{p,q}: \langle T_{A,p},T_{B,q}\rangle\to \PSL(2,\mathbb{R})$ given by
$$
T_{A,p}\mapsto 
\begin{pmatrix}
1&\mu^{p,q}_{A,B}\\
0&1
\end{pmatrix}\ ,\ \ \ \ \ 
T_{B,p}\mapsto
\begin{pmatrix}
1&0\\
-\mu^{p,q}_{A,B}&1
\end{pmatrix}\ .
$$
Moreover, the map $\rho_{A,B}^{p,q}$ lifts to $\SL(2,\mathbb{R})$ if and only if the curves in $A\cup B$ can be oriented so that their geometric and algebraic intersection numbers coincide $(\alpha_i\cdot\beta_j=i(\alpha_i,\beta_j))$.
\end{thm}

\begin{proof}
The curves in $A$ decompose the surface into $n$ copies of cylinders with height vector equal to $v^{p,q}_{A,B}$, and circumference vector equal to $Nv'^{p,q}_{A,B}$. It is obvious one can make $T_A$ act affinely when restricted to each cylinder as indicated in Definition \ref{def:Dehntwist}. Now from (\ref{eq:PerF}) the linear part of the action agrees on all cylinders. The same is true for $T_B$.
Therefore, the derivatives of $T_A$ and $T_B$ give the desired representation. When $\alpha_i\cdot\beta_j=i(\alpha_i,\beta_j)$ for $1\leq i\leq m,\ 1\leq j\leq n$, we have well-defined horizontal and vertical directions, hence a linear representation. One can see \cite[Chapter~14]{FM12} or \cite[Section~4]{Mc06} for more details. 
\end{proof}

When $\mu_{A,B}^{p,q}<2$, the corresponding graph $\mathcal{G}(A,B)$ and multiplicities are restrictive, one can see \cite[Chapter~1.4]{GHJ89} for the discussion on the graphs with norm less than $2$. When there is no multiplicity, The graphs are of type $A,D$ or $E$ \cite{leininger2004}. A few calculations also show if there are multiplicities, then only type $A$ graph appears, with multiplicity two on a  1-valence vertex and the corresponding $\mu^{p,q}_{A,B}$ equals to the norm of the type $B$ Coxeter graph. \\

\begin{rem}
 Due to the Nielson-Thurston classification, a mapping class is either periodic, reducible or Pseudo-Anosov. Thurston used this construction to construct the Pseudo-Anosov mapping class \cite{Th88}. In fact, a mapping class in this construction is periodic, reducible or Pseudo-Anosov if and only if the image under $\rho^{p,q}_{A,B}$ is elliptic, parabolic or hyperbolic (determined by the traces).   
\end{rem}

The Riemann surfaces equipped with flat structures are called flat surfaces, which can also be described analytically by a pair $(X,\omega)$, where $X$ is a closed Riemann surface, and $\omega$ is a holomorphic $1$-form on $X$. One can choose an atlas of $X$ such that $\omega$ has the form $dz$ away from zeros for some charts $z$ with transition maps given by translations, and for the neighborhood of a zero $\omega$ has the form  $w^kdw$ for some charts $w$. There is natural action of $\GL_2^+(\mathbb{R})$, given by $g(X,\omega)=(Y,g\circ\omega)$, where $Y$ is the Riemann surface such that $g\circ\omega$ is holomorphic on $Y$. If we consider the flat surfaces in the sense of Definition \ref{Def:flatsurface}, then the action is more explicit which is given by actions on polygons in $\mathbb{R}^2$. The study of the flat surfaces and their behaviors under the $\GL_2^+(\mathbb{R})$ actions has wide applications in the study of geometry, topology and dynamic systems \cite{Zorich06}, In particular, by the results of Veech, (\cite{Veech89, MTflat02}), the flat surface obtained by multi curves above has a nontrivial stabilizer group which is a discrete subgroup of $\PSL(2,\mathbb{R})$ (or $\SL(2,\mathbb{R})$) containing the group generated by affine automorphisms given by the multi-twists.
\subsection{Relation between two groups}
\begin{prop}{\cite{leininger2004}}\label{prop:orderbd}
The representation $\rho_{A,B}^{p,q}$ is faithful when $\mu_{A,B}^{p,q}\geq 2$. When $\mu_{A,B}^{p,q}<2$, the order of $\ker(\rho^{p,q}_{A,B})$ is at most $2$.
\end{prop}
\begin{proof}
When $\mu_{A,B}^{p,q}\geq 2$, the image of $\rho_{A,B}^{p,q}$ is a free group on two generators The injectivity follows from the Hopfian property of the finitely generated free group.

When $\mu_{A,B}^{p,q}<2$ with no multiplicities, see \cite[Theorem~7.3]{leininger2004}, one has a homomorphism $\delta$ from $\ker(\rho_{A,B})$ to the automorphism group of the graph preserving the
bicoloring, both the image and kernel are of order at most two and one of them is trivial. For the only case when it's not multiplicity free, by Proposition \ref{DT_property} (b), the homomorphism $\delta$ has a trivial image, hence the proof is also valid.   
\end{proof}

\begin{prop}\label{kernal}
Let $A=\{c_1,c_3,\cdots,c_{2g-1}\}, B=\{c_2,c_4,\cdots,c_{2g}\}$ be the multi curves on $\Sigma_g$ as in Figure \ref{fig:surface} with no multiplicities. We have $$\ker(\rho_{A,B})=<\iota_g>\cong \mathbb{Z}_2\ .$$ 
\end{prop}
\begin{proof}
It follows from the next lemmas.
\end{proof}

\begin{lem}\label{lem:rel0}
Let $T_{c_i}\  (1\leq i\leq 2g+1 )$ be the Dehn twists around simple closed curves as in Figure \ref{fig:surface}. Let $T_{A_g}:=T_{c_1}T_{c_2}\ldots T_{c_g},\ T_{B_g}:=T_{c_{g+1}}T_{c_{g+2}}\ldots T_{c_{2g}}$ and $\iota_g= T_{c_{2g+1}}T_{c_{2g}}\ldots T_{c_1}T_{c_1}\ldots T_{c_{2g}}T_{c_{2g+1}}$ is a hyperelliptic involution shown in Figure \ref{fig:surface}, then we have the relation
\begin{equation}
\begin{array}{cc}
   (T_{A_g}T_{B_g})^{2g+1}&=\iota_g\ , \\ 
    (T_{A_g}^{-1}(T_{A_g}T_{B_g})^{g+1})^2&=\iota_g\ .  \\
\end{array}
\end{equation}
\end{lem}
\begin{proof}
From direct computation only using braid group relations we have
\begin{equation*}
(T_{A_g}T_{B_g}T_{c_{2g+1}})^{2g+2}
   =(T_{A_g}T_{B_g})^{2g+1}\iota_g\ . 
\end{equation*}
Take a closed regular neighborhood of the union of curves $c_i$ ($1\leq i\leq 2g+1 $), as the boundary of this neighborhood consists of two nullhomotopic simple closed curves, from the chain relation \cite[Proposition~4.12]{FM12}, the left-hand side $=1$ (it is not hard to get the same relation for right Dehn twists), hence we get first equality.

For the second relation, combined with the first one, it suffices to prove
\begin{equation}\label{rel0}
T_{B_g}(T_{A_g}T_{B_g})^g=(T_{A_g}T_{B_g})^g T_{A_g}\ .    
\end{equation}
Now let $\sigma^n(T_{A_g})=T_{c_{1+n}}T_{c_{2+n}}\ldots T_{c_{g+n}}$, hence $\sigma^g(T_{A_g})=T_{B_g}$. The following relation holds in $\mathcal{B}_{2g+1}$ (since $T_{c_{i+1}}T_{A_g}T_{B_g}=T_{A_g}T_{B_g}T_{c_i}$), hence in $\Mod(\Sigma_g)$
\begin{equation}\label{rel1}
    \sigma^m(T_{A_g})T_{A_g}T_{B_g}=T_{A_g}T_{B_g}\sigma^{m-1}(T_{A_g})\ .\ \ (m\leq g)
\end{equation}
Now (\ref{rel0}) follows from (\ref{rel1})
\[
\begin{aligned}
  (T_{A_g}T_{B_g})^g T_{A_g}&=(T_{A_g}T_{B_g})^{g-1}\sigma(T_{A_g})(T_{A_g}T_{B_g})\\
  &=\cdots=\sigma^g(T_{A_g})(T_{A_g}T_{B_g})^g \\
  &=T_{B_g}(T_{A_g}T_{B_g})^g.
\end{aligned}
\]
Hence the lemma is proved.
\end{proof}

\begin{lem}\label{lem:rel1}
Same setting as before, let $T_{\tilde{A_g}}:=T_{c_1}T_{c_3}\ldots T_{c_{2g-1}},\ T_{\tilde {B_g}}:=T_{c_2}T_{c_4}\ldots T_{c_{2g}}$, we have $T_{\tilde{A_g}}T_{\tilde{B_g}}=x_g^{-1}T_{A_g}T_{B_g}x_g$, where $x_g=T_{A_{g-1}}T_{B_{g-1}}x_{g-1},\ x_1=1$.
\end{lem}
\begin{proof}
By moving the elements with a larger index to the left we can rewrite the product:
\begin{equation}
T_{\tilde{A_g}}T_{\tilde{B_g}}=T_{c_{2g-1}}T_{c_{2g}}T_{c_{2g-3}}T_{c_{2g-2}}\ldots T_{c_3}T_{c_4}T_{c_1}T_{c_2}\ .
\end{equation}
Now we prove the lemma by induction. When $g=1$, the statement is obvious, suppose it's true for $g=k$, then when $g=k+1$ we have

\begin{align*}
  T_{\tilde{A_{k+1}}}T_{\tilde{B_{k+1}}}&=T_{c_{2k+1}}T_{c_{2k+2}}T_{\tilde{A_k}}T_{\tilde{B_k}}\\
  &=T_{c_{2k+1}}T_{c_{2k+2}}(x_k^{-1}T_{A_k}T_{B_k}x_k)\\
  &=x_k^{-1}T_{c_{2k+1}}T_{c_{2k+2}}T_{A_k}T_{B_k}x_k\  (x_k\ commutes\  with\ T_{c_l}\ for\ l\geq k ) \\
  &=x_k^{-1}(T_{A_k}T_{B_k})^{-1}T_{A_k}T_{B_k}T_{c_{2k+1}}T_{c_{2k+2}}T_{A_k}T_{B_k}x_k\\
  &=x_{k+1}^{-1}T_{A_{k+1}}T_{B_{k+1}}x_{k+1}\ .
  \end{align*}
   \end{proof}
   
 From Lemma \ref{lem:rel0} and \ref{lem:rel1}, and the fact that $\iota_g$ commutes with all these Dehn twists, we have
  \begin{equation}\label{eq:kernel}
      (T_{\tilde{A_g}}T_{\tilde{B_g}})^{2g+1}=\iota_g \ .
  \end{equation}
Now from equation (\ref{eq:kernel}), $\iota_g\in \ker(\rho_{A,B})$, which has the order of at most $2$ (Proposition \ref{prop:orderbd}). Hence the Proposition \ref{kernal} is proved.

\begin{figure}[t]
\centering
\includegraphics[width=0.7\textwidth]{{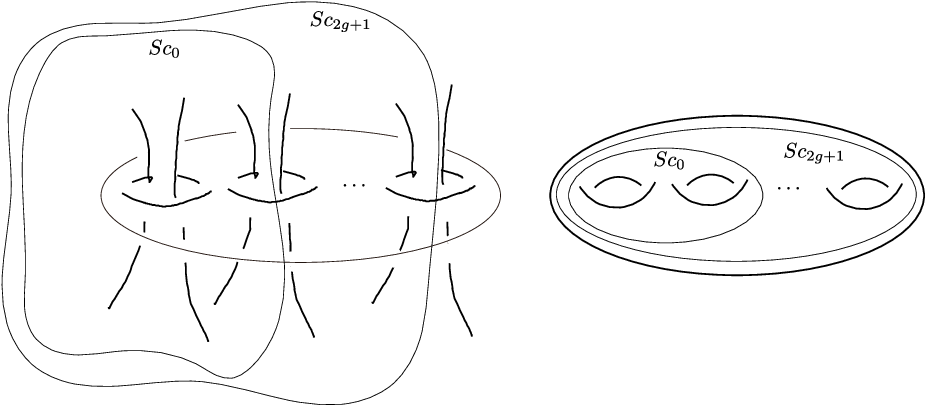}}
\caption{A Heegaard splitting represented by $S\in \Mod(\Sigma_g)$}
\label{fig:Hsplitting}
\end{figure}
 
 \begin{figure}[b]
\centering
\includegraphics[width=0.3\textwidth]{{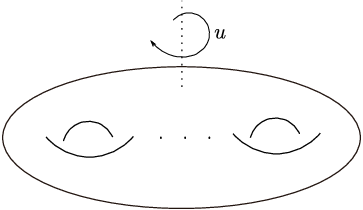}}
\caption{The element $u\in \Mod(\Sigma_g)$}
\label{fig:u}
\end{figure}
  \begin{thm} \label{Thm:Tj}
   Let $T_J=\iota_g^gT_{\tilde{A_g}}^{-1}(T_{\tilde{A_g}}T_{\tilde{B_g}})^{g+1}$, we have an injective homomorphism $\rho:\tilde{\Gamma}_{2g+1}\to \SMod(\Sigma_g)\subset \Mod(\Sigma_g)$, by sending $A_{2g+1}$, $B_{2g+1}$, $J$ to $T_{\tilde{A_g}}$, $T_{\tilde{B_g}}$, $T_J$ respectively.
   \end{thm}
   
  \begin{proof}
 The graph $\mathcal{G}(A,B)$ is of type $A$ with $2g$ vertices, and it is not hard to orient curves so that their geometric and algebraic intersection numbers coincide as indicated in Figure \ref{fig:surface} and Figure \ref{fig:Fstructure} (one may observe the difference coming from the parity of $g$). Therefore the map $\rho_{A,B}$ factors through $\SL(2,\mathbb{R})$. Now from Theorem \ref{Heckegroup}, \ref{Thconstruct}, Proposition \ref{kernal} and the fact that $\rho_{A,B}(\iota_{g})=\rho_{A,B}((T_{\tilde{A_g}}T_{\tilde{B_g}})^{2g+1})=-I$, we have $$\tilde{\Gamma}_{2g+1}\cong <T_{\tilde{A_g}},T_{\tilde{B_g}}>\xhookrightarrow{} \SMod(\Sigma_g)\ .$$
  \end{proof}
  
  \begin{cor}\label{Cor:relation}
  We have $T_J^2=\iota_g$ and $T_J=\iota_g^g (T_{\tilde{A_g}}T_{\tilde{B_g}})^{g}T_{\tilde{A_g}}$.
  \end{cor}
  \begin{proof}
    From Theorem \ref{Thm:Tj}, it suffices to prove $J^2=-I$, and $J=(-1)^g(A_{2g+1}B_{2g+1})^g A_{2g+1}$ in $\tilde{\Gamma}_{2g+1}$. The second relation follows from  $((A_{2g+1}B_{2g+1})^g A_{2g+1})^2=-I=(A_{2g+1}B_{2g+1})^{2g+1}$.
  \end{proof}
\subsection{Geometric description of $T_J$.}

Now we give a geometric description of the element $T_J$. Consider the Heegaard splitting represented by Figure \ref{fig:Hsplitting}, and denote the corresponding mapping class by $S$, $S$ maps curves $c_i$ to $c_{i+1}$ for $1\leq i\leq 2g$ and the image of $c_0,c_{2g+1}$ under $S$ is shown in the Figure \ref{fig:Hsplitting}. In fact, One can check directly $S=\prod\limits_{i=1}^{2g+1}T_{c_i}$, which is a root of center in $\mathcal{B}_{2g+2}$. We denote the horizontal rotation in Figure \ref{fig:u} by $u$, It is clear $u\in \SMod(\Sigma_g)$.

 \begin{lem} \label{Lem:Tj}
 We have  $T_J=\iota^g_guS$ in $\Mod(\Sigma_g)$.
 \end{lem}
  
  \begin{proof}
 It suffices to show the actions of these two mapping classes agree on the simple closed curves $c_i$ ($0 \leq i\leq 2g+1$) with orientations. This can be done by directly working on the surface and comparing the two actions. While the action of $\iota^g_guS$ is straightforward, the action of $T_J$ can be more complicated.  To address this, we utilize the flat structure defined in \ref{kernal} and leverage the fact that the linear action of $T_J$ is faithful (Theorem \ref{Thm:Tj}) and generally easier to describe with only a few exceptions.
 
  Directly calculation shows $|c_i|=|c_{2g+1-i}| (1\leq i\leq 2g)$, in fact we have  \cite{GHJ89} $$
  \begin{aligned}
 &v=\{sin(\frac{1}{2g+1}\pi),
  sin(\frac{3}{2g+1}\pi),\cdots,sin(\frac{2g-1}{2g+1}\pi)\}\ ,\\ &v'=\{sin(\frac{2g-1}{2g+1}\pi),sin(\frac{2g-3}{2g+1}\pi),\cdots,sin(\frac{1}{2g+1}\pi)\}\ . 
  \end{aligned}$$  Hence the corresponding flat structure has a rotation-by-$\frac{\pi}{2}$ symmetry. By the result of Veech \cite{Veech89} discussed in the previous subsection, or see \cite[Section~5]{MTflat02} and Theorem \ref{Thm:Tj}, $T_J$ is isotopic to the mapping class described in Figure \ref{fig:J} (here we give the proof for $g=3$, but the proof works for any genus similarly). Then one can check geometrically how $T_J$ moves the set of simple closed curves $c_i$ for $0\leq i\leq 2g+1$ (Figure \ref{fig:surface})  and compare with the action of $\iota^g_guS$ (It is clear how $\iota^g_guS$ acts geometrically).  If they agree on $c_i$ for all $0\leq i\leq 2g+1$  (regard as the isotopy class), then $\iota^g_guST_J^{-1}$ will fix all $c_i$, hence, by Proposition \ref{DT_property} (b), commute with all the Dehn twists $T_{c_i}$  generating $\Mod(\Sigma_g)$. The result for $g\geq 3$ now follows from the fact that the center of $\Mod(\Sigma_g)$ is trivial when $g\geq 3$. When $g=1,2$, the center is generated by the hyperelliptic involution $\iota_g$, so it suffices to check their actions agree on some curve with orientations. One can see the action on $c_0$ and $c_{2g+1}$ are the only non-obvious cases to check. For other curves, they are parallel to the edges of the cell decomposition, hence easy to see two mapping classes agree on $c_i(1<i\leq 2g)$ with orientations. We prove the case for $c_0$ and $c_{2g+1}$ in Figure \ref{fig:J}. We leave it to the reader to compare the curves in Figure \ref{fig:J} on the original surfaces.     
  \end{proof}
  
 \begin{figure}[h]
\centering
\includegraphics[width=0.7\textwidth]{{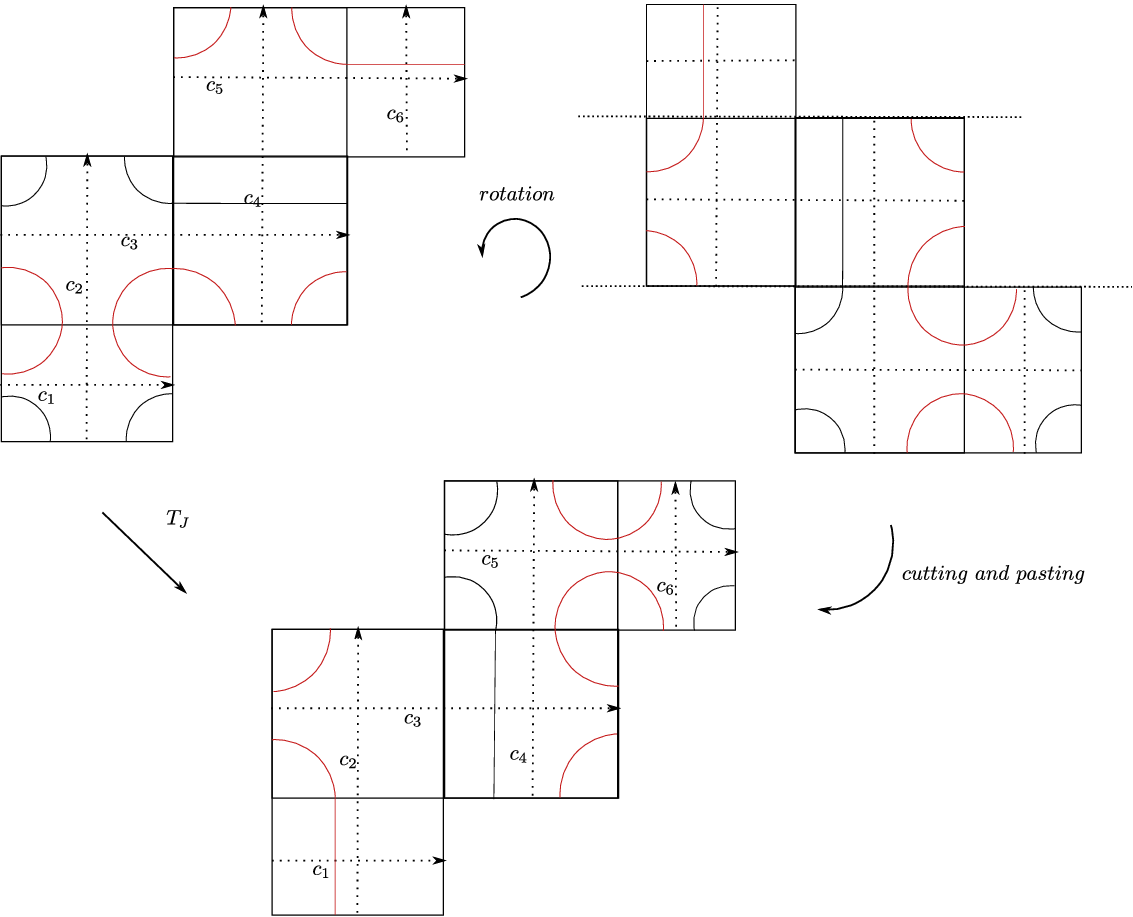}}
\caption{Action of $T_J$ on the curve $c_0$(red) and $c_{2g+1}$(black)}
\label{fig:J}
\end{figure}

 \section{Topological quantum field theory}

\subsection{Notations}
 The standard references are \cite{BHMV1,BHMV2,Roberts94},
 Let integer $p\geq 3$, $r=p-2$. We denote the color set by $I_r$. $I_r=\{0,1,\cdots,r\}$, when $r$ is even, and $I_r=\{0,2,\cdots,r-1\}$, when $r$ is odd. Let $A$ be a primitive $4p$-th root of unity when $r$ is even, and a primitive $2p$-th root of unity when $r$ is odd. For integer $i$, let $\Delta_i=(-1)^{i}[i+1]$, where $[n]=\frac{A^{2n}-A^{-2n}}{A^2-A^{-2}}$ is the quantum integer, and $\theta_i=(-1)^iA^{i(i-2)}$. We also let $\mathcal{P}_r^{\pm}=\sum_{i\in I_r}\theta_i^{\pm}\Delta_i^2$, $D_r=\sqrt{\sum_{i\in I_r}\Delta_i^2}$ and $\kappa_r=\frac{\mathcal{P}_r^+}{D_r}=\sqrt{\frac{\mathcal{P}_r^+}{\mathcal{P}_r^-}}$. $H_g$ is the genus $g$ handlebody. 

\begin{defn}
 Given a compact closed oriented three-manifold, the \textbf{skein space} $\mathcal{S}(M)$ is the vector space spanned by all the isotopy classes of framed links in $M$ modulo Kaffuman bracket relations.\\
 \begin{center}
      \includegraphics[scale=0.8]{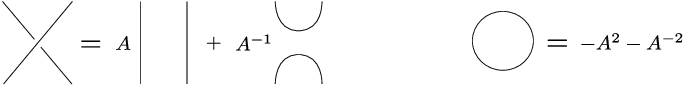}\\
 \end{center}

\end{defn}

In particular, $\mathcal{S}(S^3)=\mathbb{C}$.    

Given a framed link and an integer $r$. One can color the curves using Jones-Wenzl idempotents $f_i$ ($0\leq i\leq r$) \cite{Jones83, Wenzl87}. Let $e_i\in \mathcal{S}(H_1) $ denote the single curve winding once around the longitude, colored with $i$-th Jones-Wenzl idempotent. Let $\Omega_r=\sum_{i\in I_r}\Delta_ie_i\in \mathcal{S}(H_1)$. We use $L(\Omega_r)$ to denote the link obtained by replacing each component of $L$ by $\Omega_r$, and  $<L>_r$ to denote the number we get by resolving the framed colored link $L$ in $S^3$.

 \subsection{Invariants}
 
 Let $M_L$ be a three-dimensional manifold obtained by doing surgery along a framed link $L$ in $S^3$, we use $Z_r(M,L')$ to denote the Reshetikhin-Turaev (RT) invariant of $M$  at level $r$ with a framed colored link $L'\subset M$, we have ($\varsigma(L)$ is the signature of the linking matrix of $L$, and $m$ is the number of components of $L$)
 
 \begin{equation} \label{eq:Invat}
 Z_r(M_L,L')=(\mathcal{P}_r^{-})^{\varsigma(L)}D^{-\varsigma(L)-m-1}<L(\Omega_r)\cup L'>_r\ .
 \end{equation}
 
In particular, $Z_r(S^3)=D_r^{-1}$.
 
 \subsection{Vector space}
 The TQFT vector space of $\Sigma_g$ is constructed as the quotient of $\mathcal{S}(H_g)$ by the left kernel of any sesquilinear form induced by gluing of handlebodies \cite[Prop.~1.9]{BHMV2}. It is a finite-dimensional vector space, and we will briefly recall the explicit basis constructed in \cite{BHMV2}, also see \cite{Lickorish93} for the pure skein construction.  
 
 \begin{defn}
 Given a trivalent vertex with edges colored by $a,b,c$ from the set $I_r$. We say the vertex is \textbf{r-admissible} if the coloring satisfies:
 \begin{align*}
    & a+b+c=0\  (mod\  2)\ , \\
    & |a-b|\leq c\leq a+b\ ,\\
    & a+b+c\leq 2r\ .
 \end{align*}
 \end{defn}
 
 \begin{defn}
 An \textbf{r-admissible trivalent graph} is a labeled trivalent graph with all vertices r-admissible.
 \end{defn}

 \begin{defn}\label{Def:Basis}
  Let $\Sigma_g$ be a (closed) surface bounding the handlebody $H_g$ and  $\Gamma$ be a trivalent graph inside $H_g$ to which $H_g$ retracts. The \textbf{ TQFT vector space at level r} is spanned by all the r-admissible trivalent graphs with $\Gamma$ being the underlying graph, and all colors are from $I_r$. Figure \ref{fig:Basis} shows one possible underlying graph, admissibly coloring the graph produces one basis. We denote the basis vectors by $u_{\sigma}$, where $\sigma$ is a function from the set of edges to $I_r$. We also denote the handlebody with a basis by $H_g^{\sigma}$, and we denote the TQFT vector space by $V_r(\Sigma_g)$. it corresponds to $SU(2)$ theory when $r$ is even, and $SO(3)$ theory when $r$ is odd.
\end{defn}

Moreover we let $d_r(g):=\dim(V_r(\Sigma_g))$, the dimension is given by the formula \cite[Corollary~1.16]{BHMV2}
\begin{equation}\label{Eq:DimV}
d_r(g)=
\begin{cases}
(\frac{r+2}{2})^{g-1}\sum_{j=1}^{r+1}\left(\sin\frac{2\pi j}{2r+4}\right)^{2-2g},& \quad \text{if } r \text{ is even}\,;\\ (\frac{r+2}{4})^{g-1}\sum_{j=1}^{\frac{r+1}{2}}\left(\sin\frac{2\pi j}{r+2}\right)^{2-2g},& \quad \text{if } r \text{ is odd}\,.
\end{cases}
\end{equation}

 \begin{figure}[h]
\centering
\includegraphics[width=0.5\textwidth]{{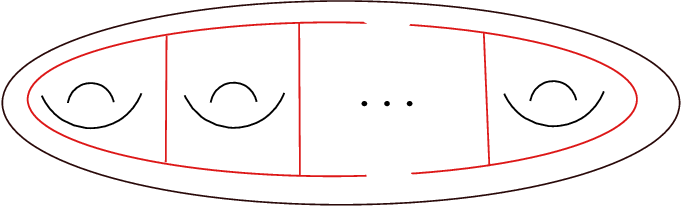}}
\caption{The underlying trivalent graph}
\label{fig:Basis}
\end{figure}
 
 One can relate different basis by doing $F$-moves locally showing in Figure \ref{fig:Fmoves}, where $F_{i,j,n}^{k,l,m}$ is the $6j$ symbol $\Sym{k}{l}{m}{i}{j}{n}$, we also denote the tetrahedron coefficient by $\Tet{*}{*}{*}{*}{*}{*}$, one can find explicit formulas in, for example, \cite{MV94}.
 \begin{figure}[h]
\centering
\includegraphics[width=0.5\textwidth]{{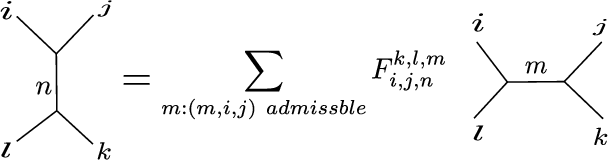}}
\caption{The F-move}
\label{fig:Fmoves}
\end{figure}\\

 For the more general setting, one can see, for example \cite{Tur10}, where TQFT is constructed from modular categories, the coloring corresponds to the simple objects in the category and the edges are morphisms. In the SU(2) and SO(3) case, the corresponding modular categories here are the TLJ-category (Temperle-Lieb-Jones category) \cite{Tur10,Wang10}, simple objects here are Jones-Wenzl projections \cite{Jones83,Wenzl87}, which have nice diagrammatic descriptions \cite{Wang10}.

 \subsection{Projective action of mapping class group}

Let $f\in \Mod(\Sigma_g)$, denote the corresponding mapping cylinder by $M_f=\Sigma_g\times [0,\frac{1}{2}]\cup_{(x,\frac{1}{2})\sim (f(x),\frac{1}{2})} \Sigma_g\times [\frac{1}{2},1]$, fix a basis for $V_r(\Sigma_g)$ as in Definition \ref{Def:Basis}, then we have a bilinear form $(\ ,\ )_{f}$ given by  $(u_{\sigma},u_{\mu})_f=Z_r(H_g^{\sigma}\cup_{id} M_{f}\cup_{id} \bar{H}_g^{\mu})$. Where $\bar{H}_g$ denotes $H_g$ with opposite orientation (hence the framing of curves inside $H_g$ will also be reversed).  It is proven in \cite{BHMV2} the bilinear form $(\ ,\ )_f$ is nondegenerate (since the vector space is constructed by quotient out the kernel), and the form $(\ ,\ )_r:=(\ ,\ )_{id}$ induced by identity mapping class is hermitian, and the basis defined in Definition \ref{Def:Basis} forms an orthogonal basis. Moreover we have \cite[Theorem~4.11]{BHMV2}:
\begin{equation}\label{eq:Renomal}
   (u_{\sigma},u_{\sigma})_r=D_r^{g-1}\frac{\prod_v\mathcal{T}_v}{\prod_{e}\Delta_{\sigma(e)}} 
\end{equation}

where $\mathcal{T}_v=\Delta_{i,j,k}$ given by evaluation of following diagram ($i,j,k$ are colors of the edges intersecting at $v$).\\

 \begin{figure}[h]
\centering
\includegraphics[width=0.3\textwidth]{{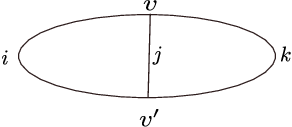}}
\caption{Theta diagram}
\label{fig:Theta}
\end{figure}

In particular, let $u_0$ denote the basis vectors with zero coloring, one has
\begin{equation}\label{eq:zeronorm}
   (u_0,u_0)_r=D_r^{g-1}. 
\end{equation}
Now the linear action can be computed by the bilinear form:  $Z_r(M_f)_{\sigma,\mu}=\frac{(u_{\sigma},u_{\mu})_f}{(u_{\mu},u_{\mu})_r}$, and we have 

\begin{equation}\label{eq:Mappingcyl}
(u_{\sigma},u_{\mu})_{f}=(Z_r(M_f)u_{\sigma},u_{\mu})_r.    
\end{equation}

The representation is projective since the signatures don't behave well with respect to the gluing see, for example \cite[Section~IV]{Tur10} and \cite[Lemma~2.8]{MR95central}.

When we pick special $A$, for example, $A=\pm ie^{\pm\frac{2\pi i}{4p}}$. The bilinear form $(\ ,\ )_r$ gives an inner product on $V_p(\Sigma_g)$. After some normalizations, our $Z_r(M_f)$ will be unitary, hence it gives a unitary projective representation of
$\Mod(\Sigma_g)$. We will simply use $Z_r(f)$ to denote $Z_r(M_f)$.\\
It is in general very hard to compute directly, for example see \cite{BWrep} for the explicit formula for the set of Dehn twists generating the mapping class group, but for some special mapping classes it is easy to write down the matrix.
When $f=T_{\gamma}$ is the right (left) Dehn twist along a curve $\gamma$, then $M_f$ can be presented by surgery on the curve $\gamma\times \{\frac{1}{2}\}\subset \Sigma_g\times I$,  which is $\mp1$ framed relative to the surface $\Sigma_g \times \{\frac{1}{2}\}$ (denoted by $\gamma^{\mp}$). If $\gamma$ bounds a disk in $H_g$, for example curves $\{c_0,c_1,c_3,\cdots,c_{2g+1}\}$, it is easy to resolve the diagrams and compute the invariant using Kirby calculus. One has 
\begin{equation}\label{oddTw0} 
   Z_r(T_{\gamma})u_{\sigma}=\theta_{\sigma(e_{\gamma})}u_{\sigma}.
\end{equation}
Where $e_{\gamma}$ is the edge transverse to the disk bounded by $\gamma$ (if such edge exists in the underlying graph for the basis).

When $f$ corresponds to a Heegaard splitting of the $S^3$, from the above discussion and the Formula (\ref{eq:Invat}), We have
$$
Z_r(S^3,L)=Z_r(S^3)<L>_r=D^{-1}<L>_r.
$$
Hence the entry of the $Z_r(M_f)$ can be computed by evaluating some tangle diagrams in $S^3$, and this is how we compute $Z_r(T_J)$. The diagram presentation for $Z_r(T_J)$ can be obtained easily from  Lemma \ref{Lem:Tj}, see Figure \ref{fig:TJdiagram} for the case when $g=2$.

There are some useful skein identities, for example in \cite{KL94,MV94}, we list some we need here.
\begin{figure}[h]
\centering
\includegraphics[width=0.35\textwidth]{{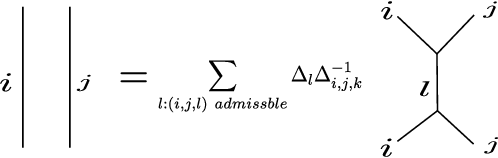}}
\caption{Partition of identity}
\label{fig:parofid}
\end{figure}

\begin{figure}[h]
\centering
\includegraphics[width=0.8\textwidth]{{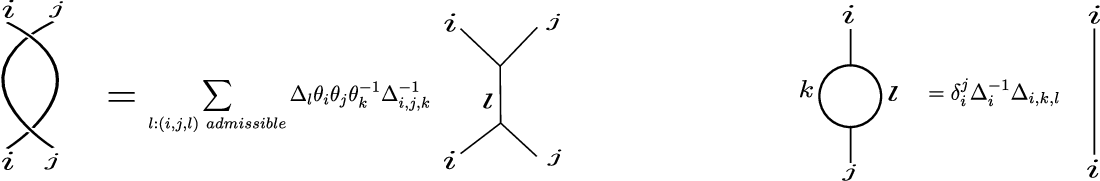}}
\caption{Two skein identities}
\label{fig:skeinidentity}
\end{figure}

Now we will briefly describe two projective actions defined in \cite{Roberts94}, also see \cite{MR95central}. One is geometric, the other is skein theoretic, we denote them by $\phi_1$ and $\phi_2$ respectively. They give the same projective representation but slightly different central extensions \cite{MR95central}, also see \cite{GManomly}.

Fix a Heegaard splitting of $S^3$ corresponding to some mapping class, for example the mapping class $T_J$ as defined in Theorem \ref{Thm:Tj} (also see Lemma \ref{Lem:Tj}), we have $S^3=H_g\cup_{T_J} \bar{H}_g=H_g\cup_{\Sigma_g\times \{0\}}(\Sigma_g\times I)\cup_{\Sigma_g\times \{1\}}\bar{H}'_g$, then $V_r(\Sigma_g)$ can be viewed as the quotient of $\mathcal{S}(H_g)$ by the kernel of the form $(\ ,\ )_{T_J}$. There are natural left and right actions of $\mathcal{S}(\Sigma_g\times I)$ on $V_r(\Sigma_g)$ by pushing the curves in $\Sigma_g\times I$ into $H_g$ and $H'_g$, and $\mathcal{S}(\Sigma_g\times I)$ is itself an algebra. Moreover it is a $*$-algebra, where the $*$ structure is induced by the map $id\times(-1)$ on $\Sigma\times I$. In particular it reverses framings relative to the surface, we use $\bar{\gamma}$ to denote the $\gamma$ under the $*$ operation. We also denote the left action by $\Add:\mathcal{S}(\Sigma_g\times I)\to \End(V_r(\Sigma_g))$, one has 
\begin{equation}
(\Add(\gamma)u_{\sigma},u_{\mu})_r=(u_{\sigma},\Add(\bar{\gamma})u_{\mu})_r.    
\end{equation} 
Now since $\Mod(\Sigma_g)$ acts naturally on $\Sigma_g\times I$, it acts on $\mathcal{S}(\Sigma_g\times I)$. It is proved in \cite{Roberts94} that $\Add$ is surjective, and moreover the kernel of $\Add$ is preserved by the action of $\Mod(\Sigma_g)$. Therefore $\Mod(\Sigma_g)$ induces automorphisms of $ \End(V_r(\Sigma_g))$, it is inner since $V_r(\Sigma_g)$ is finite dimensional. As a result, for any $f\in \Mod(\Sigma_g)$, there exists $\Phi(f)\in \End(V_r(\Sigma_g))$ such that for any $\gamma\in \mathcal{S}(\Sigma_g\times I)$ we have
\begin{equation}\label{eq:inner}
\Add(f(\gamma))=\Phi(f)\Add(\gamma)(\Phi(f))^{-1}.    
\end{equation}
Note $\Phi(f)$ is only well defined up to a constant, hence gives a projective representation. It is not hard to see $Z_r(T_J)$ satisfies (\ref{eq:inner}) for any $\gamma\in \mathcal{S}(\Sigma_g\times I)$, indeed one observe 
\begin{equation*}
(\Add(\gamma)u_{\sigma},u_{\mu})_{T_{J}}=(u_{\sigma},\Add(\overline{T_J(\gamma)})u_{\mu})_{T_{J}},
\end{equation*}
hence by (\ref{eq:Mappingcyl}), we have for any $u_{\sigma},u_{\mu}\in V_r(\Sigma_g)$,
\begin{equation}\label{eq:keyone}
\begin{aligned}
(Z_r(T_J)\Add(\gamma)u_{\sigma},u_{\mu})_r&=(Z_r(T_J)u_{\sigma},\Add(\overline{T_J({\gamma})})u_{\mu})_r\ ,\\
&=(\Add(T_J(\gamma))Z_r(T_J)u_{\sigma},u_{\mu})_r\ .
\end{aligned}    
\end{equation}
Now we will describe two actions. 

\textit{Geometric action} is defined on mapping classes which can be extended in $H_g$ or $H'_g$, we denote such mapping classes by $K,K'$ respectively. Such mapping classes directly act on the diagrams inside handlebodies, we denote the action by $\tilde{f}$.
If $f\in K$, one define $\phi_1(u_{\sigma})=\tilde{f}(u_{\sigma})$. If $f\in K'$, define $\phi_1(f)$ by 
\begin{equation}\label{eq:Geoact}
 (\phi_1(f)u_{\sigma},u_{\mu})_{T_J}=(u_{\sigma},\tilde{f}^{-1}(u_{\mu}))_{T_J}.   
\end{equation}
One can check the action $\phi_1$ satisfies (\ref{eq:inner}) \cite{Roberts94}, and since $K,K'$ generates $\Mod(\Sigma_g)$, as a projective representation we have $\phi_1=\Phi$, and in particular, we have
\begin{equation}\label{eq:keytwo}
Z_r(T_J)=c_1\phi_1(T_J)  
\end{equation} 
for some nonzero constant $c_1$ (which depends on the choice of $K\cup K'$ word representing $T_J$).
Now for Dehn twists along curves $ \{c_0,c_1,c_3,\cdots,c_{2g+1}\}$ which are in $K$. They are diagonal matrices given by 
\begin{equation}\label{oddTw1}
\phi_1(T_{\gamma})u_{\sigma}=\theta_{\sigma(e_{\gamma})}u_{\sigma}.   
\end{equation}
For Dehn twists along curves $\{c_2,\cdots,c_{2g}\}$ which are in $K'$, see Figure \ref{fig:Hsplitting} and Lemma \ref{Lem:Tj}. From (\ref{eq:Geoact}), (\ref{oddTw1}) and similar argument as in (\ref{eq:keyone}), we have for $1\leq i\leq g$
\begin{equation}\label{evenTw1}
\phi_1(T_{c_{2i}})=Z_r(T_J)\phi_1(T_{c_{2g+1-2i}})Z_r(T_J)^{-1}.
\end{equation}

\textit{Skein theoretic action} is defined on the set of all the Dehn twists, denoted by $\mathcal{D}$ , it is motivated by the surgery presentation of mapping cylinder corresponding to the Dehn twists as discussed above. Similarly, let $T_{\gamma}$ be the right (left) Dehn twist along a curve $\gamma$, then the action $\phi_2(T_{\gamma})$ is given by first cabling the curve $\gamma^{\mp}\subset\Sigma_g\times I$ by the skein element $\kappa^{\pm}_r\frac{\Omega_r}{D}$ in $\mathcal{S}(H_1)$, and then pushing it back in the handlebody $H_g$, namely $\phi_2(T_{\gamma})=\Add(\gamma^{\mp}(\kappa^{\pm}_r\frac{\Omega_r}{D}))$. Now if $\gamma$ bounds a disk in $H_g$, for example curves $\{c_0,c_1,c_3,\cdots,c_{2g+1}\}$, we have similarly
\begin{equation}\label{oddTw2} 
  \phi_2(T_{\gamma})u_{\sigma}=\theta_{\sigma(e_{\gamma})}u_{\sigma},
\end{equation}

If $\gamma$ does not bound a disk, for example curves $\{c_2,c_4,\cdots,c_{2g}\}$, one can make use of the Proposition \ref{DT_property} $(a)$, so that $\gamma$ can be map to the curves that bound a disk. For example the mapping class $T_J$, we have $T_J(c_i)=c_{2g+1-i}$ for $1\leq i \leq 2g$. Hence by (\ref{eq:inner}), we have for $1 \leq i\leq g$, 
\begin{equation}\label{evenTw2}
\phi_2(T_{c_{2i}})=Z_r(T_J)\phi_2(T_{c_{2g+1-2i}})Z_r(T_J)^{-1}=\phi_1(T_J)\phi_2(T_{c_{2g+1-2i}})\phi_1(T_J)^{-1}.    
\end{equation}

Compare equation (\ref{oddTw1}) and (\ref{oddTw2}), also from (\ref{evenTw2}), we have $\phi_1=\phi_2$ as projective representations. In particular 

\begin{equation}\label{eq:key3}
 Z_r(T_J)=c_2\phi_2(T_J).  
\end{equation}
for some nonzero constant $c_2$ (which depends on the choice of $\mathcal{D}$ word representing $T_J$).

It is not hard to see $\phi_1$ and $\phi_2$ agree on the word in $\mathcal{D}\cap (K\cup K')$ from (\ref{oddTw1})-(\ref{evenTw2}) and they are homomorphisms when restricted to $K(\cap \mathcal{D})$ or $K'(\cap \mathcal{D})$. Moreover they are unitary when the hermitian form is positive definite, see \cite{Roberts94, MR95central} for more details.
\begin{rem}
The equality in (\ref{evenTw2}) allows one to compare the projective factors between two actions, see \cite{MR95central}. Thanks to (\ref{eq:keytwo})  (\ref{eq:key3}), we can extend their result to the mapping class $T_J$, so that $c_1,c_2$ can also be determined similarly as in the proof of \cite[Lem.~2.8]{MR95central}.
\end{rem}

\begin{lem}\label{lem:cvalue}
 We have $Z_r(T_J)^2=\phi_1(\iota_g)$.
\end{lem}
\begin{proof}
Let $c$ be the $(1,1)$-entry of the matrix $(Z_r(T_J))^2$.
Since $T_J^2=\iota_g$, and
\[
(\phi_1(\iota_g)u_0,u_0)_r=(u_0,u_0)_r\ ,
\]
it suffices to prove $c=1$. We have
\[
c=\sum_{\sigma}\frac{(u_0,u_{\sigma})_{T_J}(u_{\sigma},u_0)_{T_J}}{(u_{\sigma},u_{\sigma})_r(u_0,u_0)_r}.
\]
From Lemma \ref{Lem:Tj}, and the standard calculations using second identity in Figure \ref{fig:skeinidentity}, we have $(u_{\sigma},u_0)_{T_J}=(u_0,u_{\sigma})_{T_J}=D_r^{-1}<u_{\sigma}>_r$. It is nonzero if and only if $\phi_1(\iota_g)u_{\sigma}=u_{\sigma}$. Therefore by (\ref{eq:Renomal}), we have
\[
\begin{aligned}
c&=D_r^{-2}\sum_{\sigma:\ \phi_1(\iota_g)u_{\sigma}=u_{\sigma}}\frac{<u_{\sigma}>^2_r}{(u_{\sigma},u_{\sigma})_r(u_0,u_0)_r}\ ,\\
&=D_r^{-2g}\sum_{\sigma:\ \phi_1(\iota_g)u_{\sigma}=u_{\sigma}}\prod_{e\in E'}\Delta_{\sigma(e)}\ ,
\end{aligned}
\]
where $E'$ is the set of edges (in Figure \ref{fig:Basis}) that invariant under the action of $\iota_g$. Now consider $g$ unknots placed next to each other and all colored with $\Omega_r$, there are two ways to evaluate it. Direct evaluation gives $D_r^{2g}$, the other way is to apply the partition of identity (Figure \ref{fig:parofid}) $g-1$ times, one gets a linear combination of $u_{\sigma}$. Resolving it using the second identity in Figure \ref{fig:skeinidentity} as before, one gets number $\prod_{e\in E'}\Delta_{\sigma(e)}$. Therefore $c=1$.      
\end{proof}
\begin{lem}\label{lem:anomly}
 Let $w_1$ be a $K$ word, $w_2$ be a $\mathcal{D}$ word, and $w_2w_1$ ($w_1w_2$ if $w_1$ is a $K'$ word) is a word representing the mapping class $T_J$, we have
\begin{equation}
\phi_2(w_2)\phi_1(w_1)=\kappa_r^{\tilde{\varsigma}(w_2)+e(w_2)}Z_r(T_J)\ ,
\end{equation}
where $\tilde{\varsigma}(w)=\varsigma(L_w)$ is the signature of link $L_w$ defined in \cite[Sec.~2.3]{MR95central}, and $e(w_2)$ is the exponent sum of the word $w_2$. 
\end{lem}

\begin{proof}
We have
\begin{equation*}
\begin{aligned}
(\phi_2(w_2)\phi_1(w_1)u_0,u_0)_{T_J}&=(\phi_2(w_2)\tilde{w}_1(u_0),u_0)_{T_J},\\
&=(\phi_2(w_2)u_0,u_0)_{T_J},\\
&=\kappa_r^{\tilde{\varsigma}(w_2)+e(w_2)}Z_r(M_{L_{w_2}}).
\end{aligned}
\end{equation*}

Now since the closed manifold $M_{L_{w_2}}$ is homeomorphic to the manifold 
\[
\begin{aligned}
H_g\cup_{id}M_{T_Jw^{-1}_1}\cup_{id}\bar{H}'_g &\simeq H_g\cup_{id}M_{T_Jw^{-1}_1}\cup _{id}M_{T_J}\cup_{id}\bar{H}_g\ , &\\
&\simeq H_g\cup_{id}M_{\iota_g}\cup_{id}\bar{H}_g\ , & (w^{-1}_1\in K,\ and\  T_J^2=\iota_g)\\
&\simeq H_g\cup_{id}\Sigma_g\times I\cup_{id}\bar{H}_g\ . &(\iota_g\in K\cap K')
\end{aligned}
\] 
We have $Z_r(M_{L_{w_2}})=(u_0,u_0)_r$, and the lemma follows from (\ref{eq:Mappingcyl}) and Lemma \ref{lem:cvalue}
\[
(Z_r(T_J)u_0,u_0)_{T_J}=((Z_r(T_J))^2u_0,u_0)_r=(u_0,u_0)_r\ .
\]
\end{proof}

\begin{prop}\label{Heckerel}
Let $w$ be the $\mathcal{D}$ word $(T_{\tilde{A_g}}T_{\tilde{B_g}})^{g}T_{\tilde{A_g}}$, we have
\begin{equation}
(\phi_1(T_{\tilde{A_g}})Z_r(T_J))^{2g+1}=\kappa_r^{\tilde{\varsigma}(w)+g(2g+1)}\phi_1(\iota_g).
\end{equation}
\end{prop}

\begin{proof}
From Lemma \ref{lem:cvalue}, we have $Z_r(T_J)=\phi_1(\iota_g)(Z_r(T_J))^{-1}$. Hence from (\ref{evenTw1}) or (\ref{evenTw2}), we have $Z_r(T_J)\phi_i(T_{\tilde{A_g}})Z_r(T_J)=\phi_i(T_{\tilde{B_g}})\phi_1(\iota_g)$ for $i=1,2$ ($\phi_1,\phi_2$ agrees on those Dehn twists). Now since all the mapping classes related are in $\SMod(\Sigma_g)$, they commute with $\iota_g$. Moreover since $\iota_g\in K\cap K'$, by similar but simpler argument as in Lemma \ref{lem:anomly}, we have in particular $\phi_1(\iota_g)$ commutes with $\phi_i(T_{\tilde{A_g}})$ and $\phi_i(T_{\tilde{B_g}})$ for $i=1,2$. Therefore we have 
\[
(\phi_1(T_{\tilde{A_g}})Z_r(T_J))^{2g+1}=\phi_1(\iota_g^g)\phi_2(w)Z_r(T_J).
\]
Since $T_J=\iota_g^g (T_{\tilde{A_g}}T_{\tilde{B_g}})^{g}T_{\tilde{A_g}}$ (Corollary \ref{Cor:relation}), the second equality follows from Lemma \ref{lem:cvalue} and \ref{lem:anomly}.

\end{proof}
Now we are ready for the main theorem.
\begin{thm}\label{mainThm}
We have projective (unitary) representations $h_{r}$ of $\tilde{\Gamma}_{2g+1}$ from the TQFT vector spaces $V_r(\Sigma_g)$, and  when $g\leq 2$ the representations factor through $\Gamma_{2g+1}$. In particular. When $g=1$, $h_{r}(A_1),h_{r}(J)$ gives the modular data of TLJ modular categories, when $g=2$, we have 
$\mathcal{J}_r:=h_{r}(J)$ and 
a diagonal matrix $\mathcal{T}_r:=h_r(A_3)$, satisfying the relations: 
\begin{equation}
\begin{aligned}
\mathcal{J}_r^2&=I\ ,\\
(\mathcal{T}_r\mathcal{J}_r)^5&=(\frac{\mathcal{P}_r^+}{\mathcal{P}_r^-})^2I\ .
\end{aligned}
\end{equation}
\end{thm}

\begin{proof}
Let $h_r(A_{2g+1}):=\phi_1(T_{\tilde{A_g}})$ and $h_r(J):=Z_r(T_J)$. From Lemma \ref{lem:cvalue}, Proposition \ref{Heckerel} and Theorem \ref{Thm:Tj}, it gives a projective representation of $\tilde{\Gamma}_{2g+1}$. One can get a unitary representation by specializing $A$ to some appropriate root of unity, for example, $A=\pm ie^{\pm\frac{2\pi i}{4p}}$. Normalizing the basis using (\ref{eq:Renomal}), under normalized basis, $\phi_1$ and $\phi_2$ are unitary (\cite{Roberts94}). In particular, up to a scalar, $Z_r(T_J)$ is unitary. Since its first row and column are all real numbers, the same computation as in the proof of Lemma $\ref{lem:cvalue}$ shows that $Z_r(T_J)$ is indeed unitary.   
For $g=1,2$, it's straight forward to see $\phi_1(\iota_g)=I$. The special cases in the theorem now follow from the computation of the signatures, when $g=1$, $\tilde{\varsigma}(w)=-2$, and when $g=2$, $\tilde{\varsigma}(w)=-6$.\\
\end{proof}

\begin{rem}
Theorem \ref{mainThm} and Proposition \ref{Heckerel} imply $h_r$ can be lifted to a linear representation $\tilde{h}_r$ by multiplying a root of unity $\varkappa$ in $\mathcal{T}_r$, provided $\varkappa^{2g+1}=\kappa_r^{\varsigma(w)+g(2g+1)}$.
\end{rem}

In the rest of this section, we will focus on the case where $g=2$, and we will give concrete calculations. We denote the basis in $V_r(\Sigma_2)$ by $u_{ijk}$, which means the graph in Figure \ref{fig:Basis} is colored by $i,j,k$ from left to right, and we give them the dictionary order. 
Let $\tilde{J}_{i_1j_1k_1,i_2j_2k_2}:=(u_{i_1j_1k_1},u_{i_2j_2k_2})_{T_J}$,
the matrix $\tilde{J}$ can be computed from evaluating the following diagrams.
\begin{figure}[h]
\centering
 \includegraphics[scale=0.8]{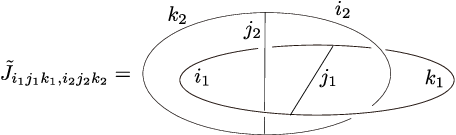}\\
 \caption{The diagram presentation for $Z_r(T_J)$ }
 \label{fig:TJdiagram}
 \end{figure}

Therefore it is easy to see $\tilde{J}$ is a symmetric real matrix. And we have
\begin{equation}
\begin{aligned}
\mathcal{J}_{i_1j_1k_1,i_2j_2k_2}&=\frac{\Delta_{i_2}\Delta_{j_2}\Delta_{k_2}}{D^2\Delta^2_{i_2j_2k_2}}\tilde{J}_{i_1j_1k_1,i_2j_2k_2}\ ,\\
\mathcal{T}_ru_{ijk}&=\theta_i\theta_j u_{i,j,k}\ .
\end{aligned}
\end{equation}
When the form $(\ ,\ )$ is positive definite one can normalize the basis, we have a unitary matrix

\begin{equation}
\mathcal{J}_{i_1j_1k_1,i_2j_2k_2}=\frac{\sqrt{\Delta_{i_1}\Delta_
{j_1}\Delta_{k_1}\Delta_{i_2}\Delta_{j_2}\Delta_{k_2}}}{D^2\Delta_{i_1,j_1,k_1}\Delta_{i_2,j_2,k_2}}\tilde{J}_{i_1j_1k_1,i_2j_2k_2}\ . 
\end{equation}
$\mathcal{T}_r$ is a diagonal matrix and the entries are all root of unities, hence always unitary.

Now we give a formula to evaluate entries of $\tilde{J}$

\begin{prop}
We have
\begin{equation}\label{eq:Jmatrix}
\tilde{J}_{i_1j_1k_1,i_2j_2k_2}=\sum^{r-2}_{l=0}\Delta_l^{-1}a^{j_1,i_2}_l\bar{a}^{k_2,i_1}_l\Tet{l}{i_2}{i_2}{j_2}{k_2}{k_2}\Tet{l}{j_1}{j_1}{k_1}{i_1}{i_1}\ .   
\end{equation}
where $a^{i,j}_k$ is the coefficient of the following recoupling formula:

\begin{center}
\includegraphics[width=0.3\textwidth]{{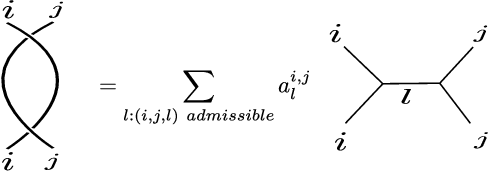}}    
\end{center}

and
\begin{equation}\label{eq:couplingcoeff}
    a^{i,j}_l=\sum_{k:(i,j,k)\
r-admissible}\Delta_k\theta_i\theta_j\theta_k^{-1}\Delta_{i,j,k}^{-1}\Sym{i}{j}{l}{j}{i}{k}\ .
\end{equation}
\end{prop}

\begin{proof}
Equation (\ref{eq:couplingcoeff}) follows from applying the identity on the left of Figure \ref{fig:skeinidentity} and a F-move. Now one can resolve two double crossings in the diagram presentation of $\tilde{J}$, and apply the identity on the right of Figure \ref{fig:skeinidentity}. The remaining diagram is two tetrahedrons connected along an edge, straightforward computations give us the Formula (\ref{eq:Jmatrix}).  
\end{proof}

\begin{rem}
 Similar diagrams (as in figure \ref{fig:TJdiagram}) also appear in \cite{LX19} as certain dualities of the Fourier transform. 
\end{rem}

Now we give the calculation result for Ising theory $r=2,A=ie^{\frac{\pi}{8}i}$ and Fibonacci theory $r=3,A=ie^{\frac{\pi}{10}i}$, which are done easily by hand and verified by using Maple software, see Appendix \ref{sec:maplecode}: 

$$ \mathcal{J}_2=\begin{pmatrix}
\frac{1}{4}&\frac{\sqrt{2}}{4}&\frac{1}{4}&\frac{\sqrt{2}}{4}&\frac{\sqrt{2}}{4}&\frac{\sqrt{2}}{4}&\frac{\sqrt{2}}{4}&\frac{1}{4}&\frac{\sqrt{2}}{4}&\frac{1}{4}\\
\frac{\sqrt{2}}{4}&\frac{1}{2}&\frac{\sqrt{2}}{4}&0&0&0&0&-\frac{\sqrt{2}}{4}   &-\frac{1}{2}&-\frac{\sqrt{2}}{4}  \\
\frac{1}{4}&\frac{\sqrt{2}}{4}&\frac{1}{4}&-\frac{\sqrt{2}}{4}&-\frac{\sqrt{2}}{4}&-\frac{\sqrt{2}}{4}&-\frac{\sqrt{2}}{4}&\frac{1}{4}&\frac{\sqrt{2}}{4}&\frac{1}{4} \\
\frac{\sqrt{2}}{4}&0&-\frac{\sqrt{2}}{4}&0&\frac{1}{2}&-\frac{1}{2}&0&-\frac{\sqrt{2}}{4}  &0&\frac{\sqrt{2}}{4}    \\
\frac{\sqrt{2}}{4}&0&-\frac{\sqrt{2}}{4}&\frac{1}{2}&0&0&-\frac{1}{2}&\frac{\sqrt{2}}{4} &0&-\frac{\sqrt{2}}{4}   \\
\frac{\sqrt{2}}{4}&0&-\frac{\sqrt{2}}{4}&-\frac{1}{2}&0&0&\frac{1}{2}&\frac{\sqrt{2}}{4}  &0&-\frac{\sqrt{2}}{4}  \\
\frac{\sqrt{2}}{4}&0&-\frac{\sqrt{2}}{4}&0 &-\frac{1}{2}&\frac{1}{2}&0&-\frac{\sqrt{2}}{4}&0&\frac{\sqrt{2}}{4}    \\
\frac{1}{4}&-\frac{\sqrt{2}}{4}&\frac{1}{4}&-\frac{\sqrt{2}}{4}&\frac{\sqrt{2}}{4}&\frac{\sqrt{2}}{4}&-\frac{\sqrt{2}}{4}&\frac{1}{4}&-\frac{\sqrt{2}}{4}&\frac{1}{4}\\
\frac{\sqrt{2}}{4}&-\frac{1}{2}&\frac{\sqrt{2}}{4}&0&0 &0&0&-\frac{\sqrt{2}}{4} &\frac{1}{2} &-\frac{\sqrt{2}}{4}\\ 
\frac{1}{4}&-\frac{\sqrt{2}}{4}&\frac{1}{4}&\frac{\sqrt{2}}{4}&-\frac{\sqrt{2}}{4}&-\frac{\sqrt{2}}{4}&\frac{\sqrt{2}}{4}&\frac{1}{4}&-\frac{\sqrt{2}}{4}&\frac{1}{4}
\end{pmatrix}$$\ 

$$
\mathcal{T}_2=\begin{pmatrix}
1&0 &0 &0 &0 &0 &0 &0 &0 &0 \\
0&e^{\frac{7\pi i}{8}}&0 &0 &0 &0 &0 &0 &0 &0 \\
0&0 &-1 &0 &0 &0 &0 &0 &0 &0 \\
0&0 &0 &e^{\frac{7\pi i}{8}} &0 &0 &0 &0 &0 &0 \\
0&0 &0 &0 &-e^{\frac{3\pi i}{4}} &0 &0 &0 &0 &0 \\
0&0 &0 &0 &0 &-e^{\frac{3\pi i}{4}} & 0&0 &0 &0 \\
0&0 &0 &0 &0 &0 &-e^{\frac{7\pi i}{8}} &0 &0 &0 \\
0&0 &0 &0 &0 &0 &0 &-1 &0 &0 \\
0&0 &0 &0 &0 &0 &0 &0 &-e^{\frac{7\pi i}{8}} &0 \\
0&0 &0 &0 &0 &0 &0 &0 &0 &1 \\
\end{pmatrix}\  
$$

$$
\mathcal{J}_3=\begin{pmatrix}
\frac{5-\sqrt{5}}{10}&\frac{\sqrt{5}}{5}&\frac{\sqrt{5}}{5}& \frac{\sqrt{5}}{5}&\frac{\sqrt{10(1+\sqrt{5})}}{10}\\
\frac{\sqrt{5}}{5}&\frac{5+\sqrt{5}}{10} & -\frac{5-\sqrt{5}}{10}& -\frac{5-\sqrt{5}}{10}&-\frac{\sqrt{10(\sqrt{5}-1)}}{10} \\
\frac{\sqrt{5}}{5}&-\frac{5-\sqrt{5}}{10} &-\frac{5-\sqrt{5}}{10} & \frac{5+\sqrt{5}}{10}& -\frac{\sqrt{10(\sqrt{5}-1)}}{10} \\
\frac{\sqrt{5}}{5}&-\frac{5-\sqrt{5}}{10} &\frac{5+\sqrt{5}}{10} &-\frac{5-\sqrt{5}}{10} & -\frac{\sqrt{10(\sqrt{5}-1)}}{10} \\
\frac{\sqrt{10(1+\sqrt{5})}}{10}& -\frac{\sqrt{10(\sqrt{5}-1)}}{10} &-\frac{\sqrt{10(\sqrt{5}-1)}}{10}  &-\frac{\sqrt{10(\sqrt{5}-1)}}{10}  & \frac{5-\sqrt{5}}{5}
\end{pmatrix}\ 
$$

$$
T_3=\begin{pmatrix}
1&0&0&0&0\\
0&e^{\frac{4}{5}\pi i}&0&0&0\\
0&0&e^{\frac{4}{5}\pi i}&0&0\\
0&0&0&e^{-\frac{2}{5}\pi i}&0\\
0&0&0&0&e^{-\frac{2}{5}\pi i}
\end{pmatrix}\ 
$$

In particular, we have
\begin{prop}
The group $\tilde{h}_3(\Gamma_5)$ is infinite.
\end{prop}
\begin{proof}
Since the projective factor is a root of unity, it suffices to work with $h_r$. We compute the element $h_3(A_5B_5^{-1})=\mathcal{J}_3\mathcal{T}_3\mathcal{J}_3\mathcal{T}^{-1}_3$. It has an eigenvalue $ \frac{1}{4}(3-\sqrt{5}+i\sqrt{2(1+3\sqrt{5})})$, which is not a root of unity, since its minimal polynomial, $1-3x+3x^2-3x^3+x^4$, is not cyclotomic.
\end{proof}

Moreover since all the entries of matrices under unnormalized basis are in a cyclotomic field $\mathbb{Q}(A)$ (only even powers of $D_r$ and $\kappa_r$ appear when $g=2$), the Galois group $\Gal(\mathbb{Q}(A);\mathbb{Q})$ naturally acts on the representations. In particular, it preserves the property that the image is finite or not. Now let $A=e^{\frac{\pi i}{r}}$ for odd $r$, By computing the trace of the image of the Pseudo-Anosov mapping class ${\rho}(A_{2g+1}B_{2g+1}^{-1})$, we have
\[
\Tr(\mathcal{J}_r\mathcal{T}_r\mathcal{J}_r\mathcal{T}^{-1}_r)=\sum_{u_{i_1j_1k_1},u_{i_2j_2k_2}}\theta_{i_1}\theta_{j_1}\theta^{-1}_{i_2}\theta^{-1}_{j_2}\mathcal{J}_{i_1j_1k_1,i_2j_2k_2}\mathcal{J}_{i_2j_2k_2,i_1j_1k_1}\ ,
\]
and
\begin{prop}
The group $\tilde{h}_r(\Gamma_5)$ is also infinite for $r=7,9,11,13$.
\end{prop}
\begin{proof}
It suffices to check  $\Tr(\mathcal{J}_r\mathcal{T}_r\mathcal{J}_r\mathcal{T}^{-1}_r)\geq d_r(2)$. Using Maple software we have the following numerical results. 
\begin{center}
\begin{tabular}{ ||c||c|c|c|c|c|c|c| } 
 \hline
 $r$ & 3 & 5 & 7 & 9 & 11 & 13\\ 
 \hline
 $d_r(g)$& 5 & 14 & 30  & 55 & 91 & 140\\ 
 \hline
 $\Tr$ & 4.24 & 10.54 & 32.16 & 102.92 & 332.49 & 1084.12\\ 
 \hline
\end{tabular}
\end{center}
We have $\Tr(\mathcal{J}_r\mathcal{T}_r\mathcal{J}_r\mathcal{T}^{-1}_r)\geq d_r(2)$ when $r=7,9,11,13$, which complete the proof.
\end{proof}
\section{Spin structure and reducibility}
The $\mathbb{Z}/2\mathbb{Z}$ spin structure on a closed surface $\Sigma_g$ of genus $g$ is cohomology class $\varphi\in H^1(UT\Sigma_g,\mathbb{Z}/2\mathbb{Z})$ which evaluates to one on the oriented fiber of the unit tangent bundle $UT\Sigma_g\to \Sigma_g$. In \cite{Johnson80}, Johnson built a one-to-one correspondence between the set of $\mathbb{Z}/2\mathbb{Z}$ spin structures on a Riemann surface $\Sigma_g$ with the set of $\mathbb{Z}/2\mathbb{Z}$ valued quadratic forms on $H_1(\Sigma_g,\mathbb{Z}/2\mathbb{Z})$ with associated intersection form on $\Sigma_g$ (i.e. $q(a+b)=q(a)+q(b)+a\cdot b$). Moreover the bijection intertwines the action of $\Mod(\Sigma_g)$ (induced by the obvious action of $Sp(2g,Z/2Z)$), and there are two orbits of the action depending on the parity of the spin structures, namely, the Arf invariant, which is an element in $\mathbb{Z}/2\mathbb{Z}$  defined by
\[Arf(q)=\sum_{i=1}^g q(x_i)q(y_i)\ ,\] 
which is independent of the choice of symplectic basis $x_i,\  y_i$  for $H_1(\Sigma_g,\mathbb{Z}/2\mathbb{Z})$. The TQFT vector space $V_r(\Sigma_g)$, for $4\mid p\  (=r+2)$, can be decomposed with respect to the spin structures on $\Sigma_g$. Recall for the action of Dehn twist along a curve on the $V_p(\Sigma_g)$ can be described by twisting the curve (make it $-1$-framed) in $\Sigma_g\times\{\frac{1}{2}\}\subset \Sigma_g\times I$, attach a skein element and then push it back in the handlebody. Now we consider the curves in $H_1(\Sigma_g,\mathbb{Z}/2\mathbb{Z})$ as the curves in $\Sigma_g\times\{\frac{1}{2}\}\subset \Sigma_g\times I$, attached with the label $r$, then there is a natural action by pushing it in the handlebody ($\Add$). It is straightforward to show the action is in fact unitary \cite[Prop ~7.5]{BHMV2}. Moreover the product induced by the algebra structure of $\mathcal{S}(\Sigma_g\times I)$ gives the set of curves a structure of finite Heisenberg group (with the associated intersection form on $\Sigma_g$ ) $\mathbb{Z}/2\mathbb{Z}\times H_1(\Sigma_g,\mathbb{Z}/2\mathbb{Z})$ \cite{BHMV2}. Which is abelian and  the characters are giving by $\mathbb{Z}/2\mathbb{Z}$ valued quadratic forms on $H_1(\Sigma_g,\mathbb{Z}/2\mathbb{Z})$. Therefore one gets the decomposition \[V_r(\Sigma_g)=\oplus_{q}V_r(\Sigma_g,q),\] where each $V_r(\Sigma_g,q)$ is the direct sum of the one-dimensional representation associate with the quadratic form $q$. Two orbits (parity) give two invariant subspaces $V_r^0, V_r^1$ under $\Mod(\Sigma_g)$ action and the associate vector spaces for spin structures of the same parity have the same dimensions, which are denoted by $d_r^0(g),d_r^1(g)$ respectively. We have $\dim(V^{\epsilon})=2^{g-1}(2^g+(-1)^{\epsilon})d^{\epsilon}_g$, for $\epsilon\in \{0,1\}$. The formula for $d_r^{\epsilon}(g)$ is given by \cite[Thm~7.16]{BHMV2} 
\begin{equation}\label{spin_dimension}
d_r^{\epsilon}(g)=2^{-2g}(d_r(g)+(\frac{r+2}{2})^{g-1}((-1)^{\epsilon}2^g-1))\ .
\end{equation}
\begin{Thm}
when $g \geq 2,\ r=4l+2\ (l\geq 1)$, the representation $\tilde{h}_r$ is reducible and has at least three irreducible summands.
\end{Thm}

\begin{proof}
We choose oriented curves $\{\alpha_i,\beta_i\}_{1\leq i\leq g}$ representing the standard symplectic
basis. The spin structure associated with the flat structure can be described by a quadratic form that assigns number $(ind_{\gamma}+1)$ (mod $2$) to $\gamma$ (for $\gamma\in \{\alpha_i,\beta_i\}_{1\leq i\leq g}$) \cite[Section~3]{Kontsevch_Zorich03}, where $ind_{\gamma}$ is the index of the curve $\gamma$. For example in Figure \ref{fig:J}, we have  $ind_{c_i}=0$ for $1\leq i\leq 6$, $ind_{c_0}=1$ and $ind_{c_7}=2$. Hence it is not hard to calculate its parity, which is equal to $\sum^{g}_{i=1}(0+1)(i-1+1)=\frac{(g+1)g}{2}$ (mod $2$). Now the group generated by $T_A, T_B$ fixes the flat structure, hence it fixes the associated spin structure, denoted by $q_{\omega}$. We have following decomposition $$V_r(\Sigma_g)=V_r(\Sigma_g,q_{\omega})\oplus (V^0_r\cap V_r^{\perp}(\Sigma_g,q_{\omega}))\oplus (V^1_r\cap V_r^{\perp}(\Sigma_g,q_{\omega})) $$ 
from the dimension counting of the TQFT vector spaces associated with spin structures (\ref{spin_dimension}), when $g\geq 2,\ l\geq 1$, dimensions for three summands are all nonzero.

\end{proof}

\begin{appendix}
\section{Maple code}\label{sec:maplecode}
Here is the code for calculating entries of $\mathcal{J},\ \mathcal{T}$ matrices. 
\begin{verbatim}
with(LinearAlgebra);
#Set up
pi := evalf(Pi);
E := evalf(exp(1));

#Unitary (replace it by E^((pi)/(p)I)for non-unitary case)
A := p -> I . (E^(1/2*I*pi/p));

#Quantum integer
nq := (p, n) -> (A(p)^(2*n) - A(p)^(-2*n))/(A(p)^2 - 1/A(p)^2);

#Quantum dimension
dq := (p, i) -> (-1)^i*nq(p, i + 1);

#Quantum factorial
fq := (p, n) -> evalf(product(nq(p, ii), ii = 1 .. n), 15);

#Admissible color
adm := (p, a, b, c) -> 'if'((a + b + c) mod 2 = 0 and 
0 <= min(a + b - c, a - b + c, -a + b + c, 2*p - 4 - a - b - c), 1, 0);

#Quantum 3j symbol
trq := (p, a, b, c) -> 'if'(adm(p, a, b, c) = 1, evalf((-1)^(1/2*a+ 1/2*b + 1/2*c)
*fq(p, -1/2*a + 1/2*b + 1/2*c)*fq(p, 1/2*a - 1/2*b+ 1/2*c)*fq(p, 1/2*a + 1/2*b
- 1/2*c)*fq(p, 1/2*a + 1/2*b + 1/2*c + 1)/(fq(p, a)*fq(p, b)*fq(p, c)), 15), 0);

syq := (p, a1, a2, a3, a4, b1, b2, b3) -> fq(p, b1 - a1)*fq(p, b1 - a2)
*fq(p, b1 - a3)*fq(p, b1 - a4)*fq(p, b2 - a1)*fq(p, b2 - a2)*fq(p, b2 - a3)
*fq(p, b2 - a4)*fq(p, b3 - a1)*fq(p, b3 - a2)*fq(p, b3 - a3)*fq(p, b3 - a4)
*sum((-1)^z*fq(p, z + 1)/(fq(p, b1 - z)*fq(p, b2 - z)*fq(p, b3 - z)
*fq(p, z - a1)*fq(p, z - a2)*fq(p, z - a3)*fq(p, z - a4)), z = max(a1, a2, 
a3, a4) .. min(b1, b2, b3));

#Quantum tetrahedron
Tet := (p, a, b, c, d, e, f) -> 'if'(adm(p, a, b, c)*adm(p, a, e, f)
*adm(p, c, d, e)*adm(p, b, d, f) = 1, evalf(syq(p, 1/2*a + 1/2*b + 1/2*c, 
1/2*a + 1/2*e + 1/2*f, 1/2*c + 1/2*d + 1/2*e, 1/2*b + 1/2*d + 1/2*f, 1/2*a 
+ 1/2*b + 1/2*d + 1/2*e, 1/2*a + 1/2*c + 1/2*d + 1/2*f, 1/2*b + 1/2*c + 1/2*e 
+ 1/2*f)/(fq(p, a)*fq(p, b)*fq(p, c)*fq(p, d)*fq(p, e)*fq(p, f)), 15), 0);

#Quantum 6j symbol
sjq := (p, a, b, i, c, d, j) -> 'if'(adm(p, i, b, c)*adm(p, i, d, a)
*adm(p, c, j, d)*adm(p, b, j, a) = 1, evalf((-1)^i*nq(p, i + 1)
*Tet(p, i, b, c, j, d, a)/(trq(p, i, a, d)*trq(p, i, b, c)), 15), 0);

Fsjq := (p, a, b, i, c, d, j) -> (sqrt((((((qd(p, j)) . (trq(p, b, c, i))) 
. (trq(p, a, d, i))) . (1/dq(p, i))) . (1/trq(p, a, b, j))) 
. (1/trq(p, c, d, j)))) . (sjq(p, a, b, i, c, d, j));

Sjq := (p, a, b, i, c, d, j) -> 'if'(adm(p, i, b, c)*adm(p, i, d, a)
*adm(p, c, j, d)*adm(p, b, j, a) = 1, evalf((-1)^(i + j)*nq(p, i + 1)
*nq(p, j + 1)*Tet(p, i, b, c, j, d, a)/(trq(p, a, b, j)*trq(p, i, a, d)
*trq(p, i, b, c)), 15), 0);

# Twists, S-matrix entries, Global dimension and central charge 
Tw := (p, a) -> (-1)^a*A(p)^(a*(a + 2));
Sm := (p, i, j) -> ((-1)^(i + j)) . (nq(p, (i + 1) . (j + 1)));
DD := p -> 'if'(p mod 2 = 1, add(nq(p, 2*i + 1)^2
, i = 0 .. 1/2*p - 3/2), add(nq(p, i + 1)^2, i = 0 .. p - 2));
P0 := p -> 'if'(p mod 2 = 1, add((Tw(p, 2.*i)) . (nq(p, 2*i + 1)^2)
, i = 0 .. 1/2*p - 3/2), add((Tw(p, i)) . (nq(p, i + 1)^2), i = 0 .. p - 2));
P1 := p -> 'if'(p mod 2 = 1, add((1/Tw(p, 2.*i)) . (nq(p, 2*i + 1)^2)
, i = 0 .. 1/2*p - 3/2), add((1/Tw(p, i)) . (nq(p, i + 1)^2), i = 0 .. p - 2));

#Quantum J coefficient.(unitary)
J := (p, f, e, d, c, b, a) -> 'if'(adm(p, a, b, c)
*adm(p, d, e, f) = 1, evalf(sqrt((-1)^(a + b + c + d + e + f)*nq(p, a + 1)
*nq(p, b + 1)*nq(p, c + 1)*nq(p, d + 1)*nq(p, e + 1)*nq(p, f + 1))
*add(Tw(p, b)*Tw(p, f)*add(Sjq(p, b, f, k, f, b, i)/Tw(p, i), i = 0 .. p - 2)
*add(Tw(p, j)*Sjq(p, d, c, k, c, d, j), j = 0 .. p - 2)
*Tet(p, k, d, d, e, f, f)*(-1)^k*Tet(p, c, a, b, b, k, c)/(Tw(p, d)*Tw(p, c)
*nq(p, k + 1)), k = 0 .. p - 2)/(trq(p, a, b, c)*trq(p, d, e, f)*DD(p)), 15), 0);

#Nonunitary
J1 := (p, f, e, d, c, b, a) -> 'if'(adm(p, a, b, c)
*adm(p, d, e, f) = 1, evalf(nq(p, a + 1)*nq(p, b + 1)*nq(p, c + 1)
*add(Tw(p, b)*Tw(p, f)*add(Sjq(p, b, f, k, f, b, i)/Tw(p, i), i = 0 .. p - 2)
*add(Tw(p, j)*Sjq(p, d, c, k, c, d, j), j = 0 .. p - 2)
*sjq(p, f, d, k, d, f, e)*trq(p, f, f, k)*trq(p, d, d, k)
*Tet(p, c, a, b, b, k, c)/(Tw(p, d)*Tw(p, c)*nq(p, k + 1)^2)
, k = 0 .. p - 2)/(trq(p, a, b, c)*trq(p, a, b, c)*DD(p)), 15), 0);

#Genus two T-matrix
T := (p, a, b, c) -> Tw(p, a)*Tw(p, b);

#computation for p=4 and 5
JJ1 := evalc(Re(Matrix(5, 5, [[J(5, 0, 0, 0, 0, 0, 0), J(5, 0, 0, 0, 0, 2, 2), 
J(5, 0, 0, 0, 2, 0, 2), J(5, 0, 0, 0, 2, 2, 0), J(5, 0, 0, 0, 2, 2, 2)], 
[J(5, 0, 2, 2, 0, 0, 0), J(5, 0, 2, 2, 0, 2, 2), J(5, 0, 2, 2, 2, 0, 2), 
J(5, 0, 2, 2, 2, 2, 0), J(5, 0, 2, 2, 2, 2, 2)], [J(5, 2, 0, 2, 0, 0, 0), 
J(5, 2, 0, 2, 0, 2, 2), J(5, 2, 0, 2, 2, 0, 2), J(5, 2, 0, 2, 2, 2, 0), 
J(5, 2, 0, 2, 2, 2, 2)], [J(5, 2, 2, 0, 0, 0, 0), J(5, 2, 2, 0, 0, 2, 2), 
J(5, 2, 2, 0, 2, 0, 2), J(5, 2, 2, 0, 2, 2, 0), J(5, 2, 2, 0, 2, 2, 2)], 
[J(5, 2, 2, 2, 0, 0, 0), J(5, 2, 2, 2, 0, 2, 2), J(5, 2, 2, 2, 2, 0, 2), 
J(5, 2, 2, 2, 2, 2, 0), J(5, 2, 2, 2, 2, 2, 2)]])));

JJ2 := evalc(Re(Matrix(10, 10, [[J(4, 0, 0, 0, 0, 0, 0), J(4, 0, 0, 0, 0, 1, 1), 
J(4, 0, 0, 0, 0, 2, 2), J(4, 0, 0, 0, 1, 0, 1), J(4, 0, 0, 0, 1, 1, 0), 
J(4, 0, 0, 0, 1, 1, 2), J(4, 0, 0, 0, 1, 2, 1), J(4, 0, 0, 0, 2, 0, 2), 
J(4, 0, 0, 0, 2, 1, 1), J(4, 0, 0, 0, 2, 2, 0)], [J(4, 0, 1, 1, 0, 0, 0), 
J(4, 0, 1, 1, 0, 1, 1), J(4, 0, 1, 1, 0, 2, 2), J(4, 0, 1, 1, 1, 0, 1), 
J(4, 0, 1, 1, 1, 1, 0), J(4, 0, 1, 1, 1, 1, 2), J(4, 0, 1, 1, 1, 2, 1), 
J(4, 0, 1, 1, 2, 0, 2), J(4, 0, 1, 1, 2, 1, 1), J(4, 0, 1, 1, 2, 2, 0)], 
[J(4, 0, 2, 2, 0, 0, 0), J(4, 0, 2, 2, 0, 1, 1), J(4, 0, 2, 2, 0, 2, 2), 
J(4, 0, 2, 2, 1, 0, 1), J(4, 0, 2, 2, 1, 1, 0), J(4, 0, 2, 2, 1, 1, 2), 
J(4, 0, 2, 2, 1, 2, 1), J(4, 0, 2, 2, 2, 0, 2), J(4, 0, 2, 2, 2, 1, 1), 
J(4, 0, 2, 2, 2, 2, 0)], [J(4, 1, 0, 1, 0, 0, 0), J(4, 1, 0, 1, 0, 1, 1), 
J(4, 1, 0, 1, 0, 2, 2), J(4, 1, 0, 1, 1, 0, 1), J(4, 1, 0, 1, 1, 1, 0), 
J(4, 1, 0, 1, 1, 1, 2), J(4, 1, 0, 1, 1, 2, 1), J(4, 1, 0, 1, 2, 0, 2), 
J(4, 1, 0, 1, 2, 1, 1), J(4, 1, 0, 1, 2, 2, 0)], [J(4, 1, 1, 0, 0, 0, 0), 
J(4, 1, 1, 0, 0, 1, 1), J(4, 1, 1, 0, 0, 2, 2), J(4, 1, 1, 0, 1, 0, 1), 
J(4, 1, 1, 0, 1, 1, 0), J(4, 1, 1, 0, 1, 1, 2), J(4, 1, 1, 0, 1, 2, 1), 
J(4, 1, 1, 0, 2, 0, 2), J(4, 1, 1, 0, 2, 1, 1), J(4, 1, 1, 0, 2, 2, 0)], 
[J(4, 1, 1, 2, 0, 0, 0), J(4, 1, 1, 2, 0, 1, 1), J(4, 1, 1, 2, 0, 2, 2), 
J(4, 1, 1, 2, 1, 0, 1), J(4, 1, 1, 2, 1, 1, 0), J(4, 1, 1, 2, 1, 1, 2), 
J(4, 1, 1, 2, 1, 2, 1), J(4, 1, 1, 2, 2, 0, 2), J(4, 1, 1, 2, 2, 1, 1), 
J(4, 1, 1, 2, 2, 2, 0)], [J(4, 1, 2, 1, 0, 0, 0), J(4, 1, 2, 1, 0, 1, 1), 
J(4, 1, 2, 1, 0, 2, 2), J(4, 1, 2, 1, 1, 0, 1), J(4, 1, 2, 1, 1, 1, 0), 
J(4, 1, 2, 1, 1, 1, 2), J(4, 1, 2, 1, 1, 2, 1), J(4, 1, 2, 1, 2, 0, 2), 
J(4, 1, 2, 1, 2, 1, 1), J(4, 1, 2, 1, 2, 2, 0)], [J(4, 2, 0, 2, 0, 0, 0), 
J(4, 2, 0, 2, 0, 1, 1), J(4, 2, 0, 2, 0, 2, 2), J(4, 2, 0, 2, 1, 0, 1), 
J(4, 2, 0, 2, 1, 1, 0), J(4, 2, 0, 2, 1, 1, 2), J(4, 2, 0, 2, 1, 2, 1), 
J(4, 2, 0, 2, 2, 0, 2), J(4, 2, 0, 2, 2, 1, 1), J(4, 2, 0, 2, 2, 2, 0)], 
[J(4, 2, 1, 1, 0, 0, 0), J(4, 2, 1, 1, 0, 1, 1), J(4, 2, 1, 1, 0, 2, 2), 
J(4, 2, 1, 1, 1, 0, 1), J(4, 2, 1, 1, 1, 1, 0), J(4, 2, 1, 1, 1, 1, 2), 
J(4, 2, 1, 1, 1, 2, 1), J(4, 2, 1, 1, 2, 0, 2), J(4, 2, 1, 1, 2, 1, 1), 
J(4, 2, 1, 1, 2, 2, 0)], [J(4, 2, 2, 0, 0, 0, 0), J(4, 2, 2, 0, 0, 1, 1), 
J(4, 2, 2, 0, 0, 2, 2), J(4, 2, 2, 0, 1, 0, 1), J(4, 2, 2, 0, 1, 1, 0), 
J(4, 2, 2, 0, 1, 1, 2), J(4, 2, 2, 0, 1, 2, 1), J(4, 2, 2, 0, 2, 0, 2), 
J(4, 2, 2, 0, 2, 1, 1), J(4, 2, 2, 0, 2, 2, 0)]])));

\end{verbatim}

\end{appendix}
\section{Data availability statement}
All data generated or analysed during this study are included in this article
\bibliographystyle{abbrv}
\bibliography{Reference}

\end{document}